\documentclass{article}
\usepackage[usenames,dvipsnames]{color}
\usepackage{amsmath,amssymb,amsfonts,amsthm,mathrsfs}
\usepackage{bm}
\usepackage{amsthm}
\usepackage{indentfirst}
\usepackage{tikz}
\usepackage{booktabs}
\usepackage{graphicx}
\usepackage{float}
\numberwithin{equation}{section}
\usepackage[right=4.5cm]{geometry}
\usepackage{marginnote}
\usepackage{esint} 

\newtheorem{theorem}{Theorem}
\newtheorem{lemma}{Lemma}

\newtheorem{corollary}{Corollary}
\newtheorem{definition}{Definition}
\newtheorem{remark}{Remark}

\usepackage{hyperref}

\newcommand{\mnote}[1]{%
  \ifmmode\text{\marginnote{\footnotesize\textcolor{blue}{#1}}}%
  \else\marginnote{\footnotesize\textcolor{blue}{#1}}%
  \fi
}

\newcommand{\relu}{\mbox{ReLU}}

\def\NN{\mathbb{N}}
\def\RR{\mathbb{R}}
\def\BB{\mathbb{B}}

\def\SS{\mathbb{S}}

\def\mal{\max\limits}
\def\mil{\min\limits}
\def\sul{\sum\limits}

\def\rr{ r}
\def\ss{ s}
\def\pc{\mathcal{P}^c}
\def\mq{\mathfrak{q}}

\def\mP{\mathcal{P}}

\def\mH{\mathcal{H}^\rr(\SS^d)}

\title{Optimal Neural Network Approximation via Empirical Least Squares with Deterministic Samples}
\author{ Xinliang Liu \thanks{School of Mathematical Sciences, Ocean University of China, Qingdao 266100, China}   \and
  Tong Mao \thanks{Shenzhen Loop Area Institute, Shenzhen 518000, P. R. China} \and
  Jinchao Xu\footnotemark[2]
}

\date{}

\begin{document}
	
\maketitle

\begin{abstract}
We develop a rigorous theory of discrete residual least-squares approximation for elliptic spectral equations $\mathfrak L_\beta u=f$ using linearized ReLU$^k$ neural networks on the sphere, where $\mathfrak L_\beta$ is a positive elliptic spectral multiplier of order $\beta$. Given a parameter set $\Theta_n=\{\theta_{j}^*\}_{j=1}^n\subset\mathbb S^d$, we approximate $u$ in the linearized network space $L_n^k(\Theta_n)$ by the discrete residual on the collocation points $\{\eta_i^*\}_{i=1}^m$
\begin{equation*}
u_{n,m}\in\arg\min_{v_n\in L_n^k(\Theta_n)}\frac1m\sum_{i=1}^m\left(f(\eta_i^*)-\mathfrak L_\beta v_n(\eta_i^*)\right)^2.
\end{equation*}
With $k>\frac{d-1}{2}+\beta$, for antipodally quasi-uniform network parameter sets and any quasi-uniform collocation points with $m\gtrsim n$, we prove that
\begin{equation*}
\|u-u_{n,m}\|_{\mathcal H^{\beta}(\mathbb S^d)}\eqsim\|f-\mathfrak L_\beta u_{n,m}\|_{\mathcal L^2(\mathbb S^d)}\lesssim n^{-\frac{r}{d}}
\begin{cases}
\|f\|_{\mathcal W^{r,p}(\mathbb S^d)},&\frac{d}{p}<r\leq \frac{d}{2},~p>2,\\
\|f\|_{\mathcal H^r(\mathbb S^d)},&r>\frac{d}{2}.
\end{cases}
\end{equation*}
We also establish a high-probability residual estimate, up to a logarithmic factor and an arbitrarily small smoothness loss, for i.i.d.\ uniformly
distributed collocation points.

The key analytical ingredient is a Bernstein inequality for linearized ReLU$^k$ network spaces. If $\underline h$ denotes the antipodal separation distance of the network parameters, then
\begin{equation*}
\|v_n\|_{\mathcal H^r(\mathbb S^d)}\lesssim\underline h^{-(r-s)}\|v_n\|_{\mathcal H^s(\mathbb S^d)},\qquad 0\leq s<r<k+\tfrac12.
\end{equation*}
%This also yields a corresponding Bernstein inequality for the residual space $\mathfrak L_\beta L_n^k$ as well as an inverse approximation theorem for the underlying network spaces.

Finally, using the positive homogeneity of ReLU$^k$ together with a constructive Sobolev extension, we obtain an affine-network interpretation on bounded Lipschitz domains in the order-zero case $\beta=0$. Numerical experiments in this setting illustrate the convergence behavior predicted by the approximation theory.
\end{abstract}

\section{Introduction}

%Deep neural networks (DNNs) form the mathematical foundation of numerous breakthroughs in artificial intelligence, serving as a powerful tool for modeling complex functions \cite{LeCun2015}. The overall performance of deep learning models hinges on several essential aspects, including expressive power, generalization behavior, computational efficiency, and robustness against perturbations.

From a mathematical perspective, the theoretical understanding of deep neural networks begins with the analysis of their fundamental components---the shallow networks that form each layer. 
Accordingly, we introduce the following function class of shallow neural networks, which serves as a basic model for the study of representation and approximation properties:
\begin{equation}\label{shallow}
    \Sigma_{n}^\sigma = \left\{
        x\mapsto\sum_{j=1}^n a_j\, \sigma(w_j \cdot x + b_j) :~ w_j \in \mathbb{R}^{d},~ a_j, b_j \in \mathbb{R}
    \right\}.
\end{equation}
The approximation rates of shallow neural networks have been extensively studied for several classes of target functions, including Barron or variation spaces, spectral Barron spaces \cite{klusowski2018approximation,siegel2022high}, and Sobolev spaces. For Barron-type classes, the dimension-independent rate $\mathcal O(n^{-1/2})$ is typically derived through probabilistic or greedy approximation arguments originating from Maurey's sampling method \cite{maurey1973type,pisier1981remarques}; see also \cite{Jones1992,barron1993universal,barron1994approximation,devore1996some,makovoz1998uniform,kurkova1,kurkova2,lewicki2004approximation,temlyakov2008greedy,Barron2008,E2019population,E2019barron,E2020representation}. For further approximation results in Barron spaces, we refer to \cite{makovoz1996random,bach2017breaking,klusowski2018approximation,xu2020finite,siegel2022sharp,siegel2022optimal,mhaskar2023tractability}, and for Sobolev spaces, where the approximation rate depends on the smoothness and dimension, to \cite{mhaskar1993approximation,petrushev1998approximation,pinkus1999approximation}. These continuous approximation results provide the benchmark for determining whether the same rates can be recovered from finitely many samples by the discrete least-squares methods.

In this work, we focus on the rectified linear unit (ReLU) activation $\relu(x)=\max\{0,x\}$ and its higher-order variants $\sigma_k(x)=(\relu(x))^k$ for $k\in\mathbb N$. Owing to the positive homogeneity of $\sigma_k$, we work directly with the normalized spherical model, in which the parameters $\displaystyle\theta=\binom{w}{b}$ lie on the sphere
$$
\SS^d=\{x\in\RR^{d+1}:~|x|=1\}.
$$
Beyond providing a natural normalization of the parameter space, the sphere also serves as the approximation domain in our analysis, where the basic ridge functions take the form $\sigma_k(\theta\cdot\eta)$ with $\theta,\eta\in\SS^d$.

Spherical domains arise naturally when the data or the underlying symmetries possess an intrinsic directional or rotational structure. For instance, spherical convolutional networks and graph-based spherical networks have been developed for signals on spheres \cite{cohen2018spherical,jiang2019spherical,defferrard2019deepsphere}, while approximation and generalization properties of neural networks on spherical domains have been studied in \cite{fang2020theory,feng2023generalization,yang2026spherical}. This provides both a natural mathematical framework for the normalized ReLU$^k$ model and a relevant setting in its own right.

For a fixed network size $n$, we describe the geometry of the parameter set $\Theta_n=\{\theta_{j,n}^*\}_{j=1}^n\subset\mathbb S^d$ through two quantities. Let $\rho$ denote the geodesic distance on $\mathbb S^d$, the fill distance
\[
h=\max_{\eta\in\mathbb S^d}\min_{1\le j\le n}\rho(\eta,\theta_{j,n}^*)
\]
measures how well the parameters cover the sphere, while the antipodal separation distance
\[
\underline h
=
\min\left\{
\min_{i\neq j}\rho(\theta_{i,n}^*,\theta_{j,n}^*),
\min_{i\neq j}\rho(-\theta_{i,n}^*,\theta_{j,n}^*)
\right\}
\]
quantifies the separation between the parameters and their antipodal points. %These two geometric quantities play different roles in the analysis: $h$ determines the approximation capability of the network space, whereas $\underline h$ governs the stability estimates developed later.

In particular, our recent work \cite{liu2025integral} showed that, for a target function of Sobolev smoothness $r$ satisfying the parity condition
\begin{equation}\label{eqn:f_evenodd_relu}
f(\eta)=(-1)^{k+1}f(-\eta),\qquad\eta\in\SS^d,
\end{equation}
one can construct an element of
\begin{equation}\label{shallowk}
L_n^k=L_n^k(\Theta_n)
=
\left\{
g\in\mathcal L^2(\SS^d):
g(\eta)=\sum_{j=1}^n a_j\sigma_k(\theta_{j,n}^*\cdot\eta),
\ a_j\in\mathbb R
\right\}
\end{equation}
whose $\mathcal L^2(\SS^d)$-approximation error is of order $\mathcal O(h^r)$.

A fundamental principle in numerical analysis asserts that the total error of a variational or learning formulation consists of an \emph{approximation error} and a \emph{quadrature (integration) error} \cite{strang1972variational,brenner2008mathematical}. 
While the quadrature error is usually negligible in classical finite element methods, recent studies have revealed that it may become dominant in neural-network-based computations, particularly for ReLU$^k$ and other non-smooth activations. 
Numerical experiments indicate that fixed high-order quadrature or uniformly sampled losses often lead to noticeable bias and instability, and achieving high accuracy requires substantially increased sampling density \cite{rivera2022quadrature,liao2023deep,magueresse2023adaptive,taylor2025stochastic}. 
This behavior agrees with classical quadrature theory: spherical ReLU$^k$ ridge functions belong to $\mathcal H^r(\mathbb S^d)$ only for $r<k+\tfrac12$, while applying an operator of order $\beta$ reduces the residual regularity threshold to $r<k+\tfrac12-\beta$. This limited residual regularity can in turn reduce the accuracy of quadrature applied to the residual loss, consistently with the deterioration observed for finitely smooth integrands in classical quadrature theory \cite{gauss1814methodus,davis1984methods}. 
To address this difficulty, several approaches have been proposed, including adaptive quadrature schemes \cite{magueresse2023adaptive}, unbiased stochastic or Monte Carlo integration \cite{rivera2022quadrature,taylor2025stochastic}, polynomial or piecewise-linear surrogate models \cite{rivera2022quadrature}, and model-aware integration strategies such as the tensor neural network (TNN) method \cite{wang2023tnn}. 

Among these strategies, the \emph{collocation method}---also referred to as \emph{discrete least-squares regression}---provides a simple way to bypass the numerical integration step. The theoretical foundations of discrete least-squares approximation have been extensively developed in the context of polynomial and parametric problems \cite{cohen2017optimal,chkifa2015discrete}, where optimal sample complexity and weighted designs are well understood. In the neural network setting, collocation has been widely employed in physics-informed neural networks (PINNs)~\cite{raissi2019physics}, random feature methods~\cite{chen2022bridging}, and extreme learning machines~\cite{HZS:2006,wang2024extreme}.

Random-design regression and randomized feature methods provide complementary
statistical context. Hsu, Kakade, and Zhang analyze random-design ordinary
and ridge regression \cite{hsu2014random}; Rudi and Rosasco establish
generalization guarantees for ridge regression with random features
\cite{rudi2017generalization}; and El Alaoui and Mahoney analyze
Nystr\"om-type kernel ridge regression using ridge leverage scores
\cite{alaoui2015fast}. Matrix concentration inequalities such as those of
Tropp \cite{tropp2012user} are a standard tool in such random Gram-matrix
analyses. These results concern statistical prediction or regularized kernel
approximation and do not imply the deterministic, unregularized spherical
sampling theorem proved here.

Let $\mathfrak L_\beta$ be a positive elliptic spectral multiplier of order $\beta\geq0$ on $\mathbb S^d$, as defined precisely in Section~\ref{sec_prelimi}, and assume that $\mathfrak L_0=I$. We consider
\begin{equation}\label{eqn:model_equation}
\mathfrak L_\beta u=f.
\end{equation}
Given a fixed network parameter set $\Theta_n=\{\theta_{j,n}^*\}_{j=1}^n$, the continuous residual least-squares approximation is
\begin{equation}
u_n^{\mathrm{cont}}\in\arg\min_{v_n\in L_n^k(\Theta_n)}\|f-\mathfrak L_\beta v_n\|_{\mathcal L^2(\mathbb S^d)}^2.
\end{equation}
We replace this continuous problem by selecting collocation points $\{\eta_i^*\}_{i=1}^m=\{\eta_{i,m}^*\}_{i=1}^{m}$ and computing
\begin{equation}\label{eqn_def_collocation}
u_{n,m}\in\arg\min_{v_n\in L_n^k(\Theta_n)}\frac1m\sum_{i=1}^m\left(f(\eta_i^*)-\mathfrak L_\beta v_n(\eta_i^*)\right)^2.
\end{equation}
Since the hidden parameters are fixed, \eqref{eqn_def_collocation} is a linear least-squares problem in the output coefficients. When $\beta=0$, it reduces to the ordinary discrete least-squares approximation problem.

Our main result shows that \eqref{eqn_def_collocation} achieves the optimal approximation rate using only $m\asymp n$ deterministic samples. More precisely, suppose that $k>\beta+\frac{d-1}{2}$, $r+\beta\leq\frac{d+2k+1}{2}$, $0\leq s<\min\{r,k+\frac12-\beta\}$, and $f$ satisfies the parity condition \eqref{eqn:f_evenodd_relu}. Then, for any $m\geq C\underline h^{-d}$ and quasi-uniform collocation points $\{\eta_i^*\}_{i=1}^m$, 
\begin{equation}\label{eqn:main_collocation_intro}
\|f-\mathfrak L_\beta u_{n,m}\|_{\mathcal H^s(\mathbb S^d)}
\lesssim\underline h^{-s}h^r
\begin{cases}
\|f\|_{\mathcal W^{r,p}(\mathbb S^d)},&\frac{d}{p}<r\leq \frac{d}{2},~p>2,\\
\|f\|_{\mathcal H^r(\mathbb S^d)},&\frac d2<r\leq\frac{d+2k+1}{2}-\beta.
\end{cases}
\end{equation}
In particular, for an antipodally quasi-uniform parameter sequence, $h\eqsim\underline h\eqsim n^{-1/d}$, so that $m\asymp n$ is sufficient and
\begin{equation*}
\|f-\mathfrak L_\beta u_{n,m}\|_{\mathcal H^s(\mathbb S^d)}
\lesssim n^{-\frac{r-s}{d}}
\begin{cases}
\|f\|_{\mathcal W^{r,p}(\mathbb S^d)},&\frac{d}{p}<r\leq \frac{d}{2},~p>2,\\
\|f\|_{\mathcal H^r(\mathbb S^d)},&r>\frac d2.
\end{cases}
\end{equation*}
Thus the deterministic collocation method achieves the optimal rate with a sample size proportional to the number of trainable output coefficients. By the elliptic norm equivalence for $\mathfrak L_\beta$ established in Section~\ref{sec_prelimi}, the corresponding solution error satisfies
\begin{equation*}
\|u-u_{n,m}\|_{\mathcal H^{s+\beta}(\mathbb S^d)}
\lesssim\|f-\mathfrak L_\beta u_{n,m}\|_{\mathcal H^s(\mathbb S^d)}.
\end{equation*}
For $\beta=0$, these estimates recover the ordinary discrete least-squares approximation results.

The stability analysis underlying \eqref{eqn:main_collocation_intro} is based on a Bernstein inequality for $L_n^k(\Theta_n)$ and the resulting inverse estimate for the residual space $\mathfrak L_\beta L_n^k(\Theta_n)$. We also establish a near-optimal high-probability residual estimate when the collocation points are sampled independently from the uniform distribution on $\mathbb S^d$. In addition, the Bernstein inequality yields an inverse approximation theorem for the underlying ReLU$^k$ trial spaces.

The equal-weight formulation in \eqref{eqn_def_collocation} differs from cubature-based discretization, where weighted sampling rules are designed to preserve polynomial exactness. Although such weights are available on the sphere, equal-weight exactness generally requires additional structure on the nodes \cite{mhaskar2001spherical,filbirMhaskar2011mz,leGiaMhaskar2009localized,bondarenko2013optimal}. In our setting, the loss integrand $|f-\mathfrak L_\beta v_n|^2$ is generally not a spherical polynomial, so existing cubature-based analyses do not directly apply. To the best of our knowledge, approximation results of the form \eqref{eqn:main_collocation_intro} have not been established for linearized neural network spaces, even under weighted or equal-weight least-squares discretizations.

We further extend the $\beta=0$ case to bounded Lipschitz domains $\Omega\subset\mathbb R^d$. A function $f$ on $\Omega$ is lifted to a spherical cap, extended to the whole sphere with the required parity, and then treated by the spherical least-squares theory. Since every step of the lifting procedure is constructive, the spherical training data and the resulting affine ReLU$^k$ approximant can be explicitly computed from the geometry of $\Omega$ and the pointwise values of $f$. Numerical examples in Section~\ref{sec:numerical_experiments} further illustrate that the deterministic discrete least-squares approximation exhibits the convergence behavior predicted by the spherical theory.

The key analytical ingredient behind the estimate above, and a result of independent interest, is a \emph{Bernstein inequality} for shallow ReLU\(^{k}\) network functions on the sphere. 
Specifically, we show that for $r>s$, any network $g$ of the form \eqref{shallowk} satisfies
\begin{equation}\label{eqn_inverse}
    \|g\|_{\mathcal{H}^r(\mathbb{S}^d)}
    \lesssim \underline{h}^{\,-(r-s)}\,
    \|g\|_{\mathcal{H}^s(\mathbb{S}^d)}.
\end{equation}

For the PDE residual analysis, the relevant space is not $L_n^k(\Theta_n)$ itself but $V_{n,\beta}=\mathfrak L_\beta L_n^k(\Theta_n)$. The elliptic norm equivalence and \eqref{eqn_inverse} immediately imply that, for every $g\in V_{n,\beta}$,
\begin{equation}\label{eqn:residual_bernstein_intro}
\|g\|_{\mathcal H^r(\mathbb S^d)}\lesssim\underline h^{-(r-s)}\|g\|_{\mathcal H^s(\mathbb S^d)},\qquad0\leq s<r<k+\tfrac12-\beta.
\end{equation}
This residual Bernstein inequality is the stability estimate used in the discrete residual least-squares analysis.

Bernstein-type inequalities control higher-order Sobolev norms by lower-order ones in finite-dimensional approximation spaces and are fundamental in approximation theory and numerical analysis \cite{bernstein1912quelques,nikolskii1975approximation,georgoulis2008inverse}. For suitable $n$-dimensional spaces $V_n\subset\mathcal H^r(\Omega)$, the classical form is
\begin{equation*}
    \|f\|_{\mathcal H^r(\Omega)}\lesssim n^{\frac{r-s}{d}}\|f\|_{\mathcal H^s(\Omega)}.
\end{equation*}
Such estimates underlie finite element, scattered data, radial basis function, kernel, and inverse approximation theory, including their manifold and spherical counterparts; related discrete norm control on spheres is provided by Marcinkiewicz--Zygmund inequalities \cite{arcangeli2014sampling,arcangeli2007extension,devore1989optimal,mhaskar2020kernel,narcowich2005sobolev,narcowich2006sobolev,hangelbroek2018inverse,marzo2007marcinkiewicz,mhaskar2001spherical}.

Bernstein-type inequalities for neural network spaces have recently attracted increasing attention, mainly in connection with variation spaces, approximation classes, and generalization \cite{siegel2022sharp,siegel2023characterization,guo2020optimal,E2019,bernstein2020distance,huang2023generalization}. The estimate \eqref{eqn_inverse} gives a sharp Bernstein inequality for linearized ReLU$^k$ networks on $\mathbb S^d$ that explicitly reflects the geometry of the network parameters and provides the stability estimate needed for our discrete least-squares analysis.

A further consequence of \eqref{eqn_inverse} is an inverse approximation theorem, which infers function smoothness from neural network approximation rates. Such converse results are classical in approximation theory and finite element analysis \cite{zhou2006approximation,maiorov2006pseudo}, and have recently been investigated for classifier models, one-dimensional ReLU approximation, and shallow ReLU networks through the Radon transform \cite{wang2023inverse,coroianu2024best,mao2024approximation}.

In the present spherical linearized setting, if
\begin{equation*}
    \inf_{v_n\in L_n^k(\Theta_n)}\|v-v_n\|_{\mathcal H^s(\mathbb S^d)}
    \lesssim n^{-\frac{r-s}{d}}
\end{equation*}
along an antipodally quasi-uniform parameter sequence, then $v\in\mathcal H^\alpha(\mathbb S^d)$ for every $\alpha<r$. Combined with the corresponding forward approximation results \cite{liu2025integral,siegel2022sharp,yang2024optimal,parhi2022near}, this yields a forward-and-inverse characterization up to an arbitrarily small loss of smoothness.

\section{Preliminaries}\label{sec_prelimi}

This section collects basic definitions and notation used throughout the paper. 
Most of the material follows the presentation in \cite{liu2025integral} and related references; it is included here for completeness and to keep the exposition self-contained.
Throughout, $d\eta$ denotes the normalized rotation-invariant surface measure on $\mathbb S^d$, so that $\int_{\mathbb S^d}1\,d\eta=1$.

\medskip

\subsection{Spherical harmonics}
We begin with standard notation for functions defined on the unit sphere
\[
    \mathbb S^d=\{x\in\mathbb R^{d+1}:|x|=1\},
\]
which serves as the primary domain for the analysis in this work. Let
$\mathbb P_m(\mathbb S^d)$ denote the space of spherical polynomials of
degree at most $m$, defined as the restriction to $\mathbb S^d$ of
polynomials in $\mathbb R^{d+1}$ of total degree at most $m$.

Let
\[
    \mathcal Y_m
    =
    \mathbb P_m(\mathbb S^d)
    \cap
    \mathbb P_{m-1}(\mathbb S^d)^\perp
\]
be the space of spherical harmonics of degree $m$, and let
\(\{Y_{m,\ell}\}_{\ell=1}^{N(m)}\) be an orthonormal basis of
$\mathcal Y_m$. The dimension of this space is
\begin{equation}
    N(m):=\dim(\mathcal Y_m)
    =\dim(\mathbb P_m(\mathbb S^d))
     -\dim(\mathbb P_{m-1}(\mathbb S^d))
    =
    \begin{cases}
        \displaystyle
        \frac{2m+d-1}{m}\binom{m+d-2}{d-1},
        &m\neq0,\\[0.8em]
        1,
        &m=0.
    \end{cases}
\end{equation}
Every $f \in \mathcal{L}^2(\mathbb{S}^d)$ admits a harmonic expansion
\begin{equation}
    f(\eta)
    = \sum_{m=0}^\infty \sum_{\ell=1}^{N(m)} 
      \widehat f(m,\ell) Y_{m,\ell}(\eta), 
      \qquad \text{a.e. } \eta \in \mathbb{S}^d,
\end{equation}
where
\(
  \widehat f(m,\ell) = 
  \langle f, Y_{m,\ell} \rangle_{\mathcal{L}^2(\mathbb{S}^d)},
  \ m \in \mathbb{N}, \ 1 \le \ell \le N(m).
\)

\medskip

\begin{definition}[Sobolev spaces on the sphere]\label{def:sobolev_sphere}
For $r > 0$ and $1 \le p \le \infty$, define the Sobolev semi-norm
\[
    |f|_{\mathcal{W}^{r,p}(\mathbb{S}^d)}
    = \|(-\Delta)^{\frac r2}f\|_{\mathcal{L}^p(\mathbb{S}^d)},
\]
where $\Delta$ is the Laplace-Beltrami operator, for which
\begin{equation}\label{eqn_def_Lap_Bel}
    (-\Delta)^{\frac r2}f=\sum_{m=0}^\infty (m(m+d-1))^{\frac{r}{2}}\Pi_mf,
\end{equation}
where $\Pi_m$ denotes the component of degree equal to $m$:
\begin{equation}
    \Pi_mf=\sul_{\ell=1}^{N(m)}\widehat f(m,\ell) Y_{m,\ell}.
\end{equation}

The corresponding Sobolev space is
\[
    \mathcal{W}^{r,p}(\mathbb{S}^d)
    = \{ f \in \mathcal{L}^p(\mathbb{S}^d) :
       \|f\|_{\mathcal{W}^{r,p}(\mathbb{S}^d)} < \infty \},
\]
with norm
\begin{equation}\label{eqn:sob_norm}
    \|f\|_{\mathcal{W}^{r,p}(\mathbb{S}^d)}
    = \|f\|_{\mathcal{L}^p(\mathbb{S}^d)}
      + |f|_{\mathcal{W}^{r,p}(\mathbb{S}^d)}.
\end{equation}
When $p=2$, we use the equivalent norm given by Parseval's theorem
\begin{equation}\label{eqn:Sob_norm_Parseval}
    |f|_{\mathcal{H}^r(\mathbb{S}^d)}^2
      = \sum_{m=0}^\infty \sum_{\ell=1}^{N(m)} m^{2r}\, \widehat f(m,\ell)^2,
    \qquad
    \|f\|_{\mathcal{H}^r(\mathbb{S}^d)}^2
      = \sum_{m=0}^\infty \sum_{\ell=1}^{N(m)} (1+m^{2r})\, \widehat f(m,\ell)^2.
\end{equation}
In this paper, we write $\mathcal{H}^r(\mathbb{S}^d) = \mathcal{W}^{r,2}(\mathbb{S}^d)$, as the two norms are equivalent. We also use the convention $\mathcal{W}^{0,p}(\mathbb{S}^d) = \mathcal{L}^{p}(\mathbb{S}^d)$ and $\mathcal{H}^0(\SS^d)=\mathcal{L}^{2}(\mathbb{S}^d)$.
\end{definition}

\medskip

\begin{definition}[Quasi-uniform and antipodally quasi-uniform point sets]\label{def:quasiuniform_sequence}
Let $d\in\mathbb{N}$. For a finite set
$\Theta_n=\{\theta_{j,n}^*\}_{j=1}^n\subset\mathbb{S}^d$, define
\begin{equation*}
    h_n:=\max_{\theta\in\mathbb{S}^d}\min_{1\le j\le n}\rho(\theta,\theta_{j,n}^*),
    \qquad
    \tilde{h}_n:=\min_{i\ne j}\rho(\theta_{i,n}^*,\theta_{j,n}^*),
\end{equation*}
and
\begin{equation*}
    \underline{h}_n:=
    \min\Big\{
      \min_{i\ne j}\rho(\theta_{i,n}^*,\theta_{j,n}^*),
      \min_{i\ne j}\rho(-\theta_{i,n}^*,\theta_{j,n}^*)
    \Big\}.
\end{equation*}
The set $\Theta_n$ is called quasi-uniform and antipodally quasi-uniform if there exists constants only dependent of $d$ such that
\begin{equation}\label{eqn:quasi-uniform}
    h_n\lesssim\tilde{h}_n,\qquad h_n\lesssim\underline{h}_n,
\end{equation}
respectively. A sequence of finite point sets
$\Theta=\{\Theta_n\}_{n\ge1}$, $\Theta_n=\{\theta_{j,n}^*\}_{j=1}^n$, is called quasi-uniform, respectively antipodally quasi-uniform, if the corresponding inequalities above hold with constants independent of $n$.
For a quasi-uniform sequence, $h_n\eqsim\tilde{h}_n\eqsim n^{-1/d}$; for an antipodally quasi-uniform sequence, $h_n\eqsim\underline{h}_n\eqsim n^{-1/d}$.
When a fixed $n$ is under consideration, we may drop the index $n$.
\end{definition}

\medskip

\subsection{Elliptic spectral multipliers}

For $\beta\geq0$ and $v\in\mathcal H^\beta(\mathbb S^d)$, we define the positive self-adjoint closed operator $\mathfrak L_\beta:\mathcal H^\beta(\mathbb S^d)\to\mathcal L^2(\mathbb S^d)$ by
\begin{equation}\label{eqn:def_operator}
\mathfrak L_\beta v=\sul_{m=0}^\infty\ell_\beta(m)\Pi_m v,
\end{equation}
where $\ell_0(m)=1$ and, for some constants $0<c_{\mathfrak L}\leq C_{\mathfrak L}<\infty$,
\begin{equation}\label{eqn:elliptic_symbol}
c_{\mathfrak L}(1+m)^\beta\leq\ell_\beta(m)\leq C_{\mathfrak L}(1+m)^\beta,\qquad m\geq0.
\end{equation}
Thus $\mathfrak L_0=I$. For every $r\geq0$, $\mathfrak L_\beta$ is an isomorphism from $\mathcal H^{r+\beta}(\mathbb S^d)$ onto $\mathcal H^r(\mathbb S^d)$, and
\begin{equation}\label{eqn:operator_equivalence}
\|\mathfrak L_\beta v\|_{\mathcal H^r(\mathbb S^d)}\eqsim\|v\|_{\mathcal H^{r+\beta}(\mathbb S^d)}.
\end{equation}
Moreover, $\mathfrak L_\beta$ commutes with every spherical harmonic multiplier, including the smooth cutoff operators introduced in Section~\ref{sec_bernstein}.

The model example is
\begin{equation*}
\mathfrak L_\beta=(I-\Delta)^{\beta/2},
\end{equation*}
where $\Delta$ is the Laplace-Beltrami operator in \eqref{eqn_def_Lap_Bel}.

More generally, if $s_\ell>0$, $\nu_\ell>0$, and $\beta=\max_{1\leq\ell\leq L}s_\ell$, then
\begin{equation*}
\mathfrak L_\beta=I+\sul_{\ell=1}^L\nu_\ell(-\Delta)^{s_\ell/2}
\end{equation*}
satisfies \eqref{eqn:elliptic_symbol}. Since $Y_{m,j}(-\eta)=(-1)^mY_{m,j}(\eta)$, the operators $\mathfrak L_\beta$ and $\mathfrak L_\beta^{-1}$ preserve parity. Consequently, if
\begin{equation*}
f(\eta)=(-1)^{k+1}f(-\eta),
\end{equation*}
then the solution $u=\mathfrak L_\beta^{-1}f$ satisfies the same identity.

If $k>\beta+\frac{d-1}{2}$, one can choose
\begin{equation*}
\frac d2<\tau<k+\frac12-\beta.
\end{equation*}
Every $v_n\in L_n^k$ then belongs to $\mathcal H^{\tau+\beta}(\mathbb S^d)$, and hence $\mathfrak L_\beta v_n\in\mathcal H^\tau(\mathbb S^d)\subset\mathcal C(\mathbb S^d)$. The pointwise residuals in \eqref{eqn_def_collocation} are therefore well defined under the assumptions of the main theorems.

\medskip

\subsection{Ultraspherical polynomials and the expansion of $\sigma_k$}

Let $\mathcal{L}^2_{w_d}([-1,1])$ be the weighted $L^2$ space equipped with the inner product
\[
    \langle f,g\rangle_{w_d}
    = \int_{-1}^1 f(t)g(t)(1-t^2)^{\frac{d-2}{2}}dt,
    \qquad
    \|f\|_{\mathcal{L}^2_{w_d}([-1,1])}
    = \langle f,f\rangle_{w_d}^{1/2}.
\]
The \emph{ultraspherical polynomials} $\{p_m\}_{m=0}^\infty$ is a standard orthogonal basis of this space
, which satisfy
\begin{equation}\label{eqn:def_ultraspherical}
    \langle p_m, p_\ell\rangle_{w_d} = 0,
    \qquad m \ne \ell, \quad m,\ell \in \mathbb{N},
\end{equation}
and can be expressed in the standard form (see \cite{szego1975orthogonal})
\begin{equation}\label{eqn:lambdam_pm}
    p_m(t)
    = \lambda_m(1-t^2)^{-\frac{d-2}{2}}
      \left(\frac{d}{dt}\right)^m
      \!\big[(1-t^2)^{m+\frac{d-2}{2}}\big],
      \qquad t \in [-1,1],
\end{equation}
for suitable normalization constants $\{\lambda_m\}_{m\ge0}$.

A useful identity (see \cite[Lemma~A.5.2]{dai2013approximation}) connects spherical and interval integrals by the surface area of $\SS^d$ and $\SS^{d-1}$:
\begin{equation}\label{eqn:int_sph_to_interval}
    \omega_{d-1}\int_{-1}^1 f(t)(1-t^2)^{\frac{d-2}{2}}dt
    = \omega_d \int_{\mathbb{S}^d} f(\theta\cdot \eta)\, d\eta,
    \qquad \theta \in \mathbb{S}^d.
\end{equation}

By \cite[Theorem~1.2.6]{dai2013approximation}, one further has the identity
\begin{equation}\label{eqn:sum_Y_nl}
    p_m(\eta\cdot \theta)
    = \sum_{\ell=1}^{N(m)} Y_{m,\ell}(\eta) Y_{m,\ell}(\theta),
\end{equation}
which relates ultraspherical polynomials to spherical harmonics.

The norms of $\{p_m\}$ follow from \eqref{eqn:int_sph_to_interval} and \eqref{eqn:sum_Y_nl}:
\begin{equation}\label{eqn_Pn_normalization}
    \|p_m\|_{\mathcal{L}^2_{w_d}([-1,1])}^2
    = \frac{\omega_d}{\omega_{d-1}} N(m).
\end{equation}
As \eqref{eqn_Pn_normalization} provides a suitable normalization, the explicit constants $\lambda_m$ are not needed in this paper. In fact, they can be recovered by comparing \eqref{eqn_Pn_normalization} with standard ultraspherical polynomial normalizations 
(see \cite[Chap.~4.3]{szego1975orthogonal}).

\medskip

The activation $\sigma_k \in \mathcal{L}^2_{w_d}([-1,1])$ admits the ultraspherical expansion
\begin{equation}\label{eqn:ultraspherical_expansion_ReLUk}
    \sigma_k = \sum_{m=0}^\infty \widehat{\sigma_k}(m)\, p_m,
\end{equation}
where
\[
    \widehat{\sigma_k}(m)
    = \frac{\langle p_m, \sigma_k\rangle_{w_d}}
           {\|p_m\|_{\mathcal{L}^2_{w_d}([-1,1])}^2}.
\]
Following \cite{schneider1967problem,bourgain2006projection,mhaskar2006weighted,bach2017breaking}, 
the nonzero coefficients form the index set
\[
    E_{\sigma_k} = 
    \{m \ge k+1 : m\equiv k+1\mod2\} \cup \{0,\ldots,k\},
\]
and for $m\geq k+1$, $m\equiv k+1\mod2$,
\begin{equation}\label{eqn:hat_sigma_large}
    \widehat{\sigma_k}(m)
    = \frac{\omega_{d-1}k!\Gamma(d/2)}{\omega_d}
      \frac{(-1)^{(m-k-1)/2}\Gamma(m-k)}
           {2^m\Gamma\!\left(\frac{m-k+1}{2}\right)
            \Gamma\!\left(\frac{m+d+k+1}{2}\right)}.
\end{equation}
Related Jacobi/ultraspherical expansions for ReLU-type zonal functions are also used in \cite{mhaskar2019reluZonal}.
The parity support and high-degree decay in \eqref{eqn:hat_sigma_large} provide the spectral input for the Bernstein argument developed in the next section.

\section{Bernstein inequality and inverse approximation theorem for ReLU$^k$ networks}\label{sec_bernstein}

For a fixed network size $n$, we write $L_n^k=L_n^k(\Theta_n)$, set $h=h_n$ and $\underline h=\underline h_n$ as in Definition~\ref{def:quasiuniform_sequence}, and suppress the subscript $n$ in $\theta_{j,n}^*$ when no confusion can arise. Thus,
\begin{equation}
h=\mal_{\eta\in\SS^d}\mil_{1\leq j\leq n}\rho(\eta,\theta_j^*),\qquad\underline h=\min\left\{\min_{i\neq j}\rho(\theta_i^*,\theta_j^*),\min_{i\neq j}\rho(-\theta_i^*,\theta_j^*)\right\}.
\end{equation}

This section presents the first main result: a Bernstein inequality for ReLU$^k$ network spaces and its inverse-approximation consequence. Unlike the classical Bernstein inequalities for polynomial and finite element spaces, the estimate compares Sobolev norms of different orders within $L_n^k$ while making the dependence on the geometry of the network parameters explicit. We first state and prove the Bernstein inequality through a high-frequency localization argument, and then derive the inverse-approximation consequence.

The following theorem gives the Bernstein inequality for network functions of the form \eqref{shallowk}. The threshold $r<k+\frac12$ is the necessary condition that $\sigma_k(\theta\cdot\circ)$ (and hence $f_n$) belongs to $\mathcal{H}^r(\SS^d)$.

\begin{theorem}\label{thm:inverse}
    Let $d,n\geq2$, $k\in\NN_0$, $0\leq s<r< k+\frac{1}{2}$, and let $\Theta_n=\{\theta_{j,n}^*\}_{j=1}^n\subset\SS^d$ be fixed with antipodal separation $\underline h>0$. With the simplified notation $\theta_j^*=\theta_{j,n}^*$, for any $f_n\in L_n^k(\Theta_n)$ the following Bernstein inequality holds:
    \begin{equation}
        \|f_n\|_{\mathcal{H}^r(\SS^d)}\lesssim \underline{h}^{-(r-s)}\|f_n\|_{\mathcal{H}^s(\SS^d)}.
    \end{equation}
    
    In particular, if $\Theta=\{\Theta_n\}_{n\ge1}$ is an antipodally quasi-uniform sequence,
    \begin{equation}
        \|f_n\|_{\mathcal{H}^r(\SS^d)}\lesssim n^{\frac{r-s}{d}}\|f_n\|_{\mathcal{H}^s(\SS^d)},
    \end{equation}
    where the implied constant is independent of $n$, $\Theta_n$, and $f_n$.

    Moreover, let $\beta\geq0$ and $\mathfrak L_\beta$ satisfy \eqref{eqn:def_operator}--\eqref{eqn:elliptic_symbol}. For $0\leq s<r< k+\frac{1}{2}-\beta$, and any $v_n\in\mathfrak{L}_\beta L_n^k(\Theta_n)$,
    \begin{equation}\label{eqn:residual_bernstein}
\|v_n\|_{\mathcal H^r(\SS^d)}\lesssim\underline h^{-(r-s)}\|v_n\|_{\mathcal H^s(\SS^d)}.
\end{equation}

\end{theorem}

The Bernstein inequality naturally yields an inverse approximation theorem, connecting the decay of approximation errors with the intrinsic smoothness of the target function. Such relationships are classical for polynomial and spline approximation; the following corollary gives the corresponding statement for ReLU$^k$ networks on the sphere.

\begin{corollary}[Inverse approximation theorem]
Let $d,k,s$ be as in Theorem~\ref{thm:inverse}, and let $v\in\mathcal H^s(\SS^d)$. If there exists some $r\in(s,k+1/2)$ such that, for some nested antipodally quasi-uniform sequence $\Theta=\{\Theta_n\}_{n\ge1}$ with $\Theta_n=\{\theta_{j,n}^*\}_{j=1}^n\subset\SS^d$, the following error estimate holds uniformly in $n$:
\begin{equation}\label{eqn_decay_inv_appr}
\inf\limits_{v_n\in L_n^k(\Theta_n)}\|v-v_n\|_{\mathcal H^s(\SS^d)}\lesssim n^{-\frac{r-s}{d}},
\end{equation}
then $v\in\mathcal H^\alpha(\SS^d)$ for all $\alpha<r$.
\end{corollary}

\begin{proof}
Fix an arbitrary nested antipodally quasi-uniform sequence $\Theta=\{\Theta_n\}_{n\ge1}$, where $\Theta_n=\{\theta_{j,n}^*\}_{j=1}^n$. Set $L_n^k=L_n^k(\Theta_n)$. As above, when applying the fixed-$n$ Bernstein inequality we suppress the subscript $n$ and write $\theta_j^*$ for $\theta_{j,n}^*$. Let $\{v_n\}_{n=1}^\infty$ be a sequence of functions satisfying
\begin{equation}\label{eqn_appr_fN}
\|v-v_n\|_{\mathcal H^s(\SS^d)}\lesssim n^{-\frac{r-s}{d}},\quad v_n\in L_n^k,\qquad n\in\NN.
\end{equation}
Since the parameter sequence is nested, $v_{2^j}-v_{2^{j+1}}\in L_{2^{j+1}}^k$. Then by Theorem~\ref{thm:inverse},
\begin{equation*}
\begin{split}
\|v_{2^j}-v_{2^{j+1}}\|_{\mathcal H^\alpha(\SS^d)}\lesssim&2^{j\frac{(\alpha-s)}{d}}\|v_{2^j}-v_{2^{j+1}}\|_{\mathcal H^s(\SS^d)}\\
\leq&\,2^{j\frac{(\alpha-s)}{d}}\big(\|v_{2^j}-v\|_{\mathcal H^s(\SS^d)}+\|v-v_{2^{j+1}}\|_{\mathcal H^s(\SS^d)}\big)\\
\lesssim&\,2^{j\frac{(\alpha-s)}{d}}(2^{-j\frac{(r-s)}{d}}+2^{-(j+1)\frac{(r-s)}{d}})\lesssim2^{-j\frac{r-\alpha}{d}}.
\end{split}
\end{equation*}
Hence, $\{v_{2^j}\}_{j=0}^\infty$ is a Cauchy sequence in $\mathcal H^\alpha(\SS^d)$ and converges to a limit $u\in\mathcal H^\alpha(\SS^d)$. Since $v$ is also the limit of $\{v_{2^j}\}_{j=0}^\infty$ in $\mathcal H^s(\SS^d)$, it follows that $v=u$, and consequently $v\in\mathcal H^\alpha(\SS^d)$.
\end{proof}

In the rest of this section, we prove Theorem~\ref{thm:inverse} by proving a stronger high-frequency estimate, stated below as Theorem~\ref{thm_inverse_hig_freq}. The full Bernstein inequality will be recovered by noticing that $\pc_{-1}f$ and $f$ only differs with a constant term.

\subsection{Localized kernels for the orthogonal expansion}
The proof relies on filtered spectral kernels associated with the spherical harmonic expansion. Applying a multiplier to the degree variable produces an integral kernel, and off-diagonal decay of this kernel is the key localization property used below.

Since the early work of L. Fejér \cite{fejer1903untersuchungen}, localized kernels of this type have been developed in several closely related settings. 
In the directions most relevant to the present proof, early systematic constructions and applications were given for trigonometric polynomials, Jacobi polynomials on the interval, and spherical polynomials in 
\cite{mhaskarPrestin2005local,mhaskar2004polynomialLocal,mhaskar2005sphereLocalized}. 
The framework was later extended to diffusion polynomial frames on quasi-metric measure spaces in 
\cite{maggioniMhaskar2008diffusion}, streamlined in the data/graph setting in 
\cite{mhaskar2018unified}, and applied to Hermite functions in 
\cite{mhaskar2017hermiteLocal}. 
Other constructions of localized polynomial kernels and frames, including the Jacobi-weight setting used below, were developed in 
\cite{petrushev2005localized,ivanov2010sub}. 

We use the following localized Jacobi-kernel estimate, obtained directly from \cite[Theorem~2.4]{petrushev2005localized} by specializing one Jacobi variable to an endpoint and adjusting the notation. For brevity, write $\|g\|_{w_d}:=\|g\|_{\mathcal L^2_{w_d}([-1,1])}$.

\begin{theorem}[Theorem~2.4, \cite{petrushev2005localized}]\label{thm_petrushev_xu_localized}
Fix $\gamma\in\NN$, and let $A\in\mathcal C^{\gamma+2d-1}(\RR)$ be supported in $[1/2,2]$. 
For $q\in\NN$, let
\begin{equation*}
    K_q(t):=\sum_{m=0}^\infty 
    A\left(\frac{m}{2^q}\right)
    \frac{p_m(1)p_m(t)}{\|p_m\|_{w_d}^2},
    \qquad -1\le t\le1.
\end{equation*}
Then
\begin{equation*}
    |K_q(t)|
    \lesssim
    \frac{2^{q\frac{d+1}{2}}\mal_{0\le\nu\le\gamma+2d-1}
    \|A^{(\nu)}\|_{\mathcal L^1([1/2,2])}}
    {(\sqrt{1+t}+2^{-q})^{\frac{d-1}{2}}
     (\sqrt{1-t}+2^{-q})^{\frac{d-1}{2}}
     (1+2^q\sqrt{1-t})^\gamma}.
\end{equation*}
The implied constant depends only on $d$ and $\gamma$.
\end{theorem}

The estimate shows that localization is controlled by the regularity of the spectral multiplier. Indeed, the decay factor $(1+2^q\sqrt{1-t})^{-\gamma}$ is obtained at the cost of differentiating the multiplier up to order $\gamma+2d-1$, as reflected in the derivative norms on the right-hand side. Consequently, sharp spectral projections are unsuitable for this argument; we instead use smooth dyadic filters.

To this end, we introduce a smooth cut-off function $\zeta$ satisfying the following conditions (see, e.g., \cite[(3.6)]{petrushev2005localized}):
\begin{eqnarray}\label{eqn_def_xi}
&&\zeta\in\mathcal{C}^\infty(\RR),\quad\zeta\geq0,\quad\mathrm{supp}(\zeta)\subset[1/2,2],\\
&&\zeta(t)>c_1>0,\qquad t\in[3/5,5/3],\\
&&\zeta(t)+\zeta(2t)=1,\quad t\in[1/2,1].
\end{eqnarray}
A standard consequence of these properties is the partition of unity relation
\begin{equation}\label{eqn_zeta_sum1}
    1=\sul_{q=0}^\infty\zeta(2^{-q}m),\qquad m\geq1.
\end{equation}

\medskip
Set
\begin{equation}
    v_0=p_0+\sul_{m=1}^{\infty}\zeta(m)p_m,
    \qquad
    v_q=\sul_{m=1}^{\infty}\zeta\big(\frac{m}{2^q}\big)p_m,
    \qquad q\geq1.
\end{equation}
Define the smooth cutoff operators $\mathcal{P}_{\mq}$ by
\begin{equation}
    \begin{split}
        \mathcal{P}_{\mq}f(\theta)
        :=&\Big[f*\Big(\sul_{q=0}^\mq v_q\Big)\Big](\theta)=\int_{\SS^d}f(\eta)\Big(p_0(\theta\cdot\eta)
        +\sul_{q=0}^\mq\sul_{m=1}^{\infty}\zeta\big(\frac{m}{2^q}\big)p_m(\theta\cdot\eta)\Big)d\eta\\
        =&\Pi_0f(\theta)+\sul_{m=1}^{\infty}\Big[\sul_{q=0}^\mq\zeta\big(\frac{m}{2^q}\big)\Big]
        \Pi_mf(\theta),\quad\mq\in\NN_0,\\
        \mathcal{P}_{-1}f(\theta):=&\Pi_0f(\theta)=\int_{\SS^d}f(\eta)d\eta.
    \end{split}
\end{equation}
By \eqref{eqn_def_xi} and \eqref{eqn_zeta_sum1}, for $m\geq1$,
\begin{equation*}
    \sul_{q=0}^\mq\zeta\big(\frac{m}{2^q}\big)=\left\{\begin{array}{ll}
        1, &\quad m\leq2^{\mq},  \\
        \zeta\big(\frac{m}{2^{\mq}}\big), & \quad2^{\mq}+1\leq m\leq2^{\mq+1},\\
        0, &\quad m\geq2^{\mq+1}+1,
    \end{array}\right.
\end{equation*}
Thus $\mathcal P_{\mq}$ reproduces degrees $m\leq2^\mq$, transitions smoothly for $2^\mq<m\leq2^{\mq+1}$, and vanishes for $m\geq2^{\mq+1}+1$.

The corresponding high-frequency operator is denoted by 
$\mathcal{P}^c_{\mq}f=f-\mathcal{P}_{\mq}f$. 
Intuitively, $\mathcal{P}^c_{\mq}$ extracts the part of $f$ containing spherical harmonics with degree larger than $2^\mq$.

\medskip
We can now state the key high-frequency estimate underlying Theorem~\ref{thm:inverse}.

\begin{theorem}\label{thm_inverse_hig_freq}
    Let $d,n\geq2$, $\mq\geq-1$, $k\in\NN_0$, $0\leq s<r< k+\frac{1}{2}$, and let $\{\theta_j^*\}_{j=1}^n\subset\SS^d$ have antipodal separation $\underline h>0$. There exists a constant $c_*>0$ such that whenever $2^\mq\underline h\leq c_*$, every $f_n\in L_n^k$ satisfies
    \begin{equation}\label{eqn_inver_highfre}
        \|\pc_\mq f_n\|_{\mathcal{H}^r(\SS^d)}\lesssim 
        \underline{h}^{-(r-s)}\|\pc_\mq f_n\|_{\mathcal{H}^s(\SS^d)}.
    \end{equation}
    Moreover, let $\beta\geq0$ and $\mathfrak L_\beta$ satisfy \eqref{eqn:def_operator}--\eqref{eqn:elliptic_symbol}. For $0\leq s<r< k+\frac{1}{2}-\beta$, and any $v_n\in V_{n,\beta}=\mathfrak{L}_\beta L_n^k(\Theta_n)$,
\begin{equation}\label{eqn:filtered_residual_bernstein}
\|\mathcal P_{\mathfrak q}^c v_n\|_{\mathcal H^{r}(\SS^d)}\lesssim\underline h^{-(r-s)}\|\mathcal P_{\mathfrak q}^c v_n\|_{\mathcal H^{s}(\SS^d)}.
\end{equation}
\end{theorem}

The proof proceeds in three steps. First, we fill the vanishing high-degree coefficients of $\sigma_k$ by introducing an auxiliary kernel $\phi_k$. Next, for
$$
g_n=\sul_{j=1}^n\tilde a_j\phi_k(\theta_j^*\cdot\circ),
$$
we represent the Sobolev seminorms of $\pc_\mq g_n$ as quadratic forms. Localization of the dyadic kernels $v_{\alpha,q}$ then yields diagonal dominance and the filtered Bernstein inequality under an ordinary separation condition. Finally, adjoining the antipodal parameter set transfers this estimate from $\phi_k$-networks back to ReLU$^k$ networks.
\subsection{Filling the spectral gaps of $\sigma_k$}

Recall from \eqref{eqn:ultraspherical_expansion_ReLUk}--\eqref{eqn:hat_sigma_large} that the nonzero high-degree coefficients of $\sigma_k$ lie in the parity class $m\equiv k+1\mod 2$, while the complementary parity class vanishes. To prove Theorem~\ref{thm_inverse_hig_freq}, we fill these spectral gaps by defining an auxiliary kernel $\phi_k$ with ultraspherical expansion
\begin{equation}
    \phi_k(t)=\sul_{m=0}^\infty\widehat{\phi_k}(m)p_m(t),
\end{equation}
where the coefficients are given by
\begin{equation}\label{eqn:def_phi_k}
    \widehat{\phi_k}(m)
    =\left\{\begin{array}{ll}
    |\widehat{\sigma_k}(m)|, & m\leq k,  \\[0.3em]
    \displaystyle
    \left(\frac{\omega_{d-1}}{\omega_d}
           \frac{k!\Gamma(d/2)}{2^{k+1}\sqrt{\pi}}\right)
    \frac{\Gamma\!\left(\frac{m-k}{2}\right)}
         {\Gamma\!\left(\frac{m+d+k+1}{2}\right)}, & m\geq k+1.
\end{array}\right.
\end{equation}

\medskip
For later use, we introduce a convenient notation. 
For any $\alpha\ge0$, define
\begin{equation}\label{eqn:theta}
    \xi_\alpha(t)
    =\left(\frac{\omega_{d-1}}{\omega_d}
            \frac{k!\Gamma(d/2)}{2^{k+1}\sqrt{\pi}}\right)^2
      t^{2\alpha}
      \left(\frac{\Gamma\!\left(\frac{t-k}{2}\right)}
                 {\Gamma\!\left(\frac{t+d+k+1}{2}\right)}\right)^2,
      \qquad t\ge k+1,
\end{equation}
and extend $\xi_\alpha$ smoothly to $[0,\infty)$ so that $\xi_\alpha(m)=m^{2\alpha}\widehat{\phi_k}(m)^2$ for all $m\ge 0$.

The asymptotic behavior of $\xi_\alpha$ plays a crucial role in our analysis. 
In particular, the following estimate—proved in \cite[Lemma~3.1]{liu2025integral}—describes its polynomial decay and will be frequently used below:
\begin{equation}\label{eqn_xi_decay}
    (-1)^\nu\xi_\alpha^{(\nu)}(t)\eqsim 
    t^{2\alpha-(d+2k+1)-\nu},\qquad \nu=0,1,\dots.
\end{equation}

\medskip
This construction of $\phi_k$ effectively “fills out’’ the spectral gaps of $\sigma_k$, 
allowing us to treat $\phi_k$ as a smooth surrogate kernel whose spectral coefficients decay in a controlled manner. 
It enables the subsequent application of localization and diagonal dominance arguments in the proof of the Bernstein inequality.

\subsection{Quadratic-form representation of $\mathcal{H}^\alpha$-seminorms}

For $\tilde a=(\tilde a_1,\dots,\tilde a_n)^\top\in\RR^n$, set
\[
g_n=\sul_{j=1}^n\tilde a_j\phi_k(\theta_j^*\cdot\circ).
\]
We now rewrite the Sobolev seminorms of $\pc_\mq g_n$ as quadratic forms convenient for the localization argument.
By Definition~\ref{def:sobolev_sphere}, the $\mathcal{H}^\alpha$-seminorm of $\pc_\mq g_n$ can be written as
\begin{equation}\label{eqn:Sobolev_norm_gn}
    \begin{split}
        |\pc_\mq g_n|_{\mathcal{H}^\alpha(\SS^d)}^2
        =&\sul_{m=2^\mq}^\infty \sul_{\ell=1}^{N(m)}
        \Big[\sul_{q=\mq+1}^\infty \zeta\big(\frac{m}{2^q}\big)\Big]^2
        m^{2\alpha}\widehat{\phi_k}(m)^2
        \Big(\sul_{j=1}^n\tilde{a}_jY_{m,\ell}(\theta_j^*)\Big)^2\\
        =&\sul_{m=2^\mq}^\infty
        \Big[\sul_{q=\mq+1}^\infty \zeta\big(\frac{m}{2^q}\big)\Big]^2
        m^{2\alpha}\widehat{\phi_k}(m)^2
        (\tilde{a}^\top P(m)\tilde{a}),
    \end{split}
\end{equation}
where $\{P(m)\}_{m=0}^\infty$, following the notation in \cite{liu2025integral}, are the $n\times n$ positive semidefinite matrices defined by
\begin{equation}
    (P(m))_{i,j}
    =\sul_{\ell=1}^{N(m)}Y_{m,\ell}(\theta_i^*)Y_{m,\ell}(\theta_j^*)
    =p_m(\theta_i^*\cdot\theta_j^*),
\end{equation}
and the last equality follows from \eqref{eqn:sum_Y_nl}.

\medskip
The term $\Big[\sul_{q=\mq+1}^\infty \zeta\big(\frac{t}{2^q}\big)\Big]^2$ appearing above admits a simpler equivalent form. 
Since $\mathrm{supp}(\zeta)\subset[1/2,2]$ and $\zeta(t)+\zeta(2t)=1$ for $t\in[1/2,1]$, one obtains the following piecewise expression:
\begin{equation}
    \begin{split}
        \Big[\sul_{q=\mq+1}^\infty \zeta\big(\frac{t}{2^q}\big)\Big]^2
        =&
        \left\{
        \begin{array}{ll}
        0, & t<2^\mq,\\[0.3em]
        \zeta\big(\frac{t}{2^{\mq+1}}\big)^2, & 2^\mq\le t\le2^{\mq+1}, \\[0.3em]
        1, & t>2^{\mq+1},
        \end{array}
        \right.\\
        =&\sul_{q=\mq+1}^\infty \varphi\big(\frac{t}{2^q}\big),
    \end{split}
\end{equation}
where the auxiliary function $\varphi$ is defined by
\begin{equation}
    \varphi(t)=\zeta(t)^2+2\zeta\!\big(\tfrac{t}{2}\big)\zeta(t),
    \qquad t\in\RR.
\end{equation}

\medskip
With this simplification, the Sobolev seminorms can be rewritten in a compact quadratic form:
\begin{equation}\label{eqn_sob_norm_s_r}
    \begin{split}
        |\pc_\mq g_n|_{\mathcal{H}^s(\SS^d)}^2
        &=\sul_{q=\mq+1}^\infty 
        \sul_{m=2^{q-1}}^{2^{q+1}}
        \varphi\big(\frac{m}{2^q}\big)\xi_\ss(m)
        \Big(\tilde{a}^\top P(m)\tilde{a}\Big)
        =\sul_{q=\mq+1}^\infty\Big(\tilde{a}^\top Q_\ss(q)\tilde{a}\Big),\\
        |\pc_\mq g_n|_{\mathcal{H}^r(\SS^d)}^2
        &=\sul_{q=\mq+1}^\infty 
        \sul_{m=2^{q-1}}^{2^{q+1}}
        \varphi\big(\frac{m}{2^q}\big)\xi_\rr(m)
        \Big(\tilde{a}^\top P(m)\tilde{a}\Big)
        =\sul_{q=\mq+1}^\infty\Big(\tilde{a}^\top Q_\rr(q)\tilde{a}\Big),
    \end{split}
\end{equation}
where, for $\alpha=\rr,\ss$,
\begin{equation}
    Q_\alpha(q)
    =\big(v_{\alpha,q}(\theta_i^*\cdot\theta_j^*)\big)_{i,j=1}^n,
\end{equation}
and the kernel functions $v_{\alpha,q}$ are given by
\begin{equation*}
    v_{\alpha,q}(t)
    =\sul_{m=2^{q-1}}^{2^{q+1}}
    \varphi\big(\frac{m}{2^q}\big)\xi_\alpha(m)p_m(t),
    \qquad t\in[-1,1].
\end{equation*}

\medskip
This representation expresses the Sobolev seminorms entirely in terms of the quadratic forms
$\tilde{a}^\top Q_\alpha(q)\tilde{a}$, 
where the matrices $Q_\alpha(q)$ capture the localized frequency interactions of the neural network functions on the sphere. 
This reformulation provides the foundation for the diagonal-dominance argument developed in the next subsection.

\subsection{Bernstein inequality for filtered $\phi_k$-networks}

In this subsection, we establish a structural property of the matrices $\{Q_\alpha(q)\}_{q=1}^\infty$, from which the desired Bernstein inequality for filtered $\phi_k$-networks follows.
The argument hinges on the localization estimate for the kernels $\{v_{\alpha,q}\}_{q=1}^\infty$ introduced in the previous subsection.

\begin{lemma}\label{lem_local_poly}
For any $\gamma\in\NN$,
        \begin{equation}
            |v_{\alpha,q}(t)|\lesssim\frac{2^{-q(\frac{d+1}{2}+2k+\gamma-2\alpha)}}{(\sqrt{1+t}+2^{-q})^{\frac{d-1}{2}}(\sqrt{1-t}+2^{-q})^{\frac{d-1}{2}+\gamma}}.
        \end{equation}
\end{lemma}

\begin{proof}
    We first note that $\frac{\omega_d}{\omega_{d-1}}p_m(1)=\|p_m\|_{w_d}^2$, so $v_{\alpha,q}$ can be rewritten as
    \begin{equation}
        v_{\alpha,q}(t)=\frac{\omega_d}{\omega_{d-1}}\sul_{m=2^{q-1}}^{2^{q+1}}\varphi\big(\frac{m}{2^q}\big)\xi_\alpha(m)\frac{p_m(1)p_m(t)}{\|p_m\|_{w_d}^2}.
    \end{equation}
    We apply Theorem \ref{thm_petrushev_xu_localized} and consider the auxiliary function
    \begin{equation*}
        l(u)=\varphi(u)\xi_\alpha(2^qu).
    \end{equation*}
    For any fixed integer $\kappa\ge1$, using Leibniz’s rule and then Cauchy’s bound, we obtain
    \begin{equation}
        \begin{split}
            \int_{\frac{1}{2}}^2|l^{(\kappa)}(u)|du
            =&\int_{\frac{1}{2}}^2\Bigg|\sul_{\nu=0}^\kappa\binom{\kappa}{\nu}\varphi^{(\kappa-\nu)}(u)\Big(\frac{d}{du}\Big)^\nu(\xi_\alpha(2^qu))\Bigg|du\\
            \leq&2^\kappa\int_{\frac{1}{2}}^2\max\limits_{0\leq\nu\leq\kappa}\big|\varphi^{(\kappa-\nu)}(u)2^{q\nu}\xi_\alpha^{(\nu)}(2^qu)\big|du
            \lesssim\max\limits_{0\leq\nu\leq\kappa}2^{q\nu}\int_{\frac{1}{2}}^2|\xi_\alpha^{(\nu)}(2^qu)|du.
        \end{split}
    \end{equation}
    By \eqref{eqn_xi_decay}, $|\xi_\alpha^{(\nu)}(2^qu)|\lesssim2^{-q(d+2k+1-2\alpha+\nu)}$, hence
    \begin{equation}
        \|l^{(\kappa)}\|_{\mathcal{L}^1}\lesssim2^{-q(d+2k+1-2\alpha)},
    \end{equation}
    with an implied constant independent of $q$. 
    Applying that estimate yields, for any $\gamma\in\NN$,
    \begin{equation*}
        \begin{split}
            |v_{\alpha,q}(t)|
            \lesssim&\frac{2^{q\frac{d+1}{2}}}{(\sqrt{1+t}+2^{-q})^{\frac{d-1}{2}}(\sqrt{1-t}+2^{-q})^{\frac{d-1}{2}}(1+2^q\sqrt{1-t})^{\gamma}}\|l^{(\gamma+2d-1)}\|_{\mathcal{L}^1}\\
            \lesssim&\frac{2^{-q(\frac{d+1}{2}+2k+\gamma-2\alpha)}}{(\sqrt{1+t}+2^{-q})^{\frac{d-1}{2}}(\sqrt{1-t}+2^{-q})^{\frac{d-1}{2}+\gamma}}.
        \end{split}
    \end{equation*}
\end{proof}

Let
\begin{equation*}
        \tilde{h}:=\min\limits_{i\neq j}\rho(\theta_i^*,\theta_j^*).
\end{equation*}
Following \cite{liu2025integral}, we partition the index set $\{i:1\leq i\leq n,~i\neq j\}$ according to the distances of $\theta_i^*$ from $\theta_j^*$ and $-\theta_j^*$:
\begin{equation*}
    \{i:1\leq i\leq n,~i\neq j\}=\mathcal{I}_{-1,j}\cup\bigcup\limits_{p=0}^{\lceil\log_2\left(\frac{\pi}{2\tilde h}\right)\rceil}\left(\mathcal{I}_{p,j,+}\cup\mathcal{I}_{p,j,-}\right),
\end{equation*}
where $\mathcal{I}_{-1,j}:=\left\{i:\rho(\theta_i^*,-\theta_j^*)<\tilde h\right\}$ and, for $p=0,1,\dots$,
$$\mathcal{I}_{p,j,+}:=\{i:2^p\tilde h\leq\rho(\theta_i^*,\theta_j^*)<2^{p+1}\tilde h\},\quad\mathcal{I}_{p,j,-}:=\{i:2^p\tilde h\leq\rho(\theta_i^*,-\theta_j^*)<2^{p+1}\tilde h\}.$$
A standard measure argument yields
$$\#\mathcal{I}_{-1,j}\lesssim1,\quad\#\mathcal{I}_{p,j,+}\lesssim2^{pd},\quad\#\mathcal{I}_{p,j,-}\lesssim2^{pd},$$
with constants depending only on $d$.

Using
$$\sqrt{1-\theta_i^*\cdot \theta_j^*}=\sqrt{1-\cos(\rho(\theta_i^*,\theta_j^*))}=\sqrt{2}\sin\left(\frac{\rho(\theta_i^*,\theta_j^*)}{2}\right),\qquad \theta_i^*\cdot \theta_j^*\geq0,$$
we have $\sqrt{1-\theta_i^*\cdot \theta_j^*}\eqsim\rho(\theta_i^*,\theta_j^*)$.
By Lemma \ref{lem_local_poly}, it follows that
\begin{equation*}
    \left|\left(Q_{\alpha}(q)\right)_{i,j}\right|\lesssim\frac{2^{-q(\frac{d+1}{2}+2k+\gamma-2\alpha)}}{(\rho(\theta_i^*,-\theta_j^*)+2^{-q})^{\frac{d-1}{2}}(\rho(\theta_i^*,\theta_j^*)+2^{-q})^{\frac{d-1}{2}+\gamma}},
\end{equation*}
and hence, for $q\geq\log_2\big({\tilde h^{-1}}\big)$,
\begin{equation*}
    \left|\left(Q_{\alpha}(q)\right)_{i,j}\right|\lesssim\frac{2^{-q(\frac{d+1}{2}+2k+\gamma-2\alpha)}}{(\rho(\theta_i^*,-\theta_j^*)+2^{-q})^{\frac{d-1}{2}}\rho(\theta_i^*,\theta_j^*)^{\frac{d-1}{2}+\gamma}},\qquad i\neq j.
\end{equation*}
Summing over $i\ne j$ with the above partition, for $\gamma>\frac{d+1}{2}$,
\begin{equation}\label{eqn_Q_ij_sum}
    \begin{split}
        \sul_{i\neq j}\left|\left(Q_{\alpha}(q)\right)_{i,j}\right|
        \lesssim& \sul_{p=0}^{\lceil\log_2\left(\frac{\pi}{2\tilde h}\right)\rceil}\sul_{i\in\mathcal{I}_{p,j,+}\cup\mathcal{I}_{p,j,-}}\frac{2^{-q(\frac{d+1}{2}+2k+\gamma-2\alpha)}}{(\rho(\theta_i^*,-\theta_j^*)+2^{-q})^{\frac{d-1}{2}}\rho(\theta_i^*,\theta_j^*)^{\frac{d-1}{2}+\gamma}}\\
        &+\sul_{i\in\mathcal{I}_{-1,j}}\frac{2^{-q(\frac{d+1}{2}+2k+\gamma-2\alpha)}}{(\rho(\theta_i^*,\theta_j^*)+2^{-q})^{\frac{d-1}{2}+\gamma}}\\
        \lesssim&\sul_{p=0}^{\lceil\log_2\left(\frac{\pi}{2\tilde h}\right)\rceil}\sul_{i\in\mathcal{I}_{p,j,+}\cup\mathcal{I}_{p,j,-}}(2^p\tilde h)^{-\left(\frac{d-1}{2}+\gamma\right)}2^{-q(\frac{d+1}{2}+2k+\gamma-2\alpha)}\\
        &+\sul_{i\in\mathcal{I}_{-1,j}}2^{-q(2k+1+\gamma-2\alpha)}\\
        \lesssim&\sul_{p=0}^{\lceil\log_2\left(\frac{\pi}{2\tilde h}\right)\rceil}2^{pd}(2^p\tilde h)^{-\left(\frac{d-1}{2}+\gamma\right)}2^{-q(\frac{d+1}{2}+2k+\gamma-2\alpha)}+2^{-q(2k+1+\gamma-2\alpha)}\\
        \lesssim&2^{-q(\frac{d+1}{2}+2k+\gamma-2\alpha)}\tilde h^{-(\frac{d-1}{2}+\gamma)}+2^{-q(2k+1+\gamma-2\alpha)}.
    \end{split}
\end{equation}

On the other hand, by \eqref{eqn_Pn_normalization} and \eqref{eqn_xi_decay},
\begin{equation*}
    p_m(1)=N(m)\eqsim2^{q(d-1)},\quad\xi_\alpha(m)\eqsim2^{-q(d+2k+1-2\alpha)},\qquad2^{q-1}\leq m\leq2^{q+1}.
\end{equation*}
Together with $\zeta(t)>c_1>0$ for $t\in[3/5,5/3]$, this implies
\begin{equation}\label{eqn_Qii}
    \begin{split}
        \left|\left(Q_{\alpha}(q)\right)_{i,i}\right|
        =&\sul_{m=2^{q-1}}^{2^{q+1}}\varphi\big(\frac{m}{2^q}\big)\xi_\alpha(m)p_m(1)
        \geq\sul_{m=\frac{3}{5}2^{q}}^{\frac{5}{3}2^{q}}c_1^2\xi_\alpha(m)p_m(1)\\
        \gtrsim&2^q\cdot2^{-q(d+2k+1-2\alpha)}\cdot2^{q(d-1)}
        =2^{-q(2k+1-2\alpha)},\\
        \left|\left(Q_{\alpha}(q)\right)_{i,i}\right|
        =&\sul_{m=2^{q-1}}^{2^{q+1}}\varphi\big(\frac{m}{2^q}\big)\xi_\alpha(m)p_m(1)
        \leq\sul_{m=2^{q-1}}^{2^{q+1}}\xi_\alpha(m)p_m(1)\\
        \lesssim&2^q\cdot2^{-q(d+2k+1-2\alpha)}\cdot2^{q(d-1)}
        =2^{-q(2k+1-2\alpha)},\qquad q=1,2,\dots.
    \end{split}
\end{equation}

Combining \eqref{eqn_Q_ij_sum} and \eqref{eqn_Qii}, there exists $C_5>0$ such that for $q\geq\lceil\log_2(C_5\tilde{h}^{-1})\rceil$,
\begin{equation*}
    \frac{1}{2}\left(Q_{\alpha}(q)\right)_{j,j}\geq\sul_{i\neq j}\left|\left(Q_{\alpha}(q)\right)_{i,j}\right|.
\end{equation*}
For notational simplicity, throughout the sequel, we let
\begin{equation}
    \tilde q:=\lceil\log_2(C_5\tilde{h}^{-1})\rceil,
\end{equation}
where $C_5$ may be increased without affecting the preceding estimate. We choose it so that $c_*:=C_5/2\geq\pi$. Then $2^\mq\tilde h\leq c_*$ implies $\mq+1\leq\tilde q$, and the following diagonal dominance relations hold for $q\geq\tilde q$:
\begin{equation}\label{eqn:dig_dom_matrix}
    \begin{split}
        &Q_{\alpha}(q)-\left(\frac{1}{2}\left(Q_{\alpha}(q)\right)_{i,i}\right)I_{n\times n}=\left(\frac{1}{2}\left(Q_{\alpha}(q)\right)_{i,i}\right)I_{n\times n}+\left(Q_{\alpha}(q)-\left(Q_{\alpha}(q)\right)_{i,i}I_{n\times n}\right)\succeq O_{n\times n},\\
        &\left(\frac{3}{2}\left(Q_{\alpha}(q)\right)_{i,i}\right)I_{n\times n}-Q_{\alpha}(q)=\left(\frac{1}{2}\left(Q_{\alpha}(q)\right)_{i,i}\right)I_{n\times n}+\left(\left(Q_{\alpha}(q)\right)_{i,i}I_{n\times n}-Q_{\alpha}(q)\right)\succeq O_{n\times n}.
    \end{split}
\end{equation}

Given the diagonal dominance of $Q_\alpha(q)$, we can now prove the Bernstein inequality for $\pc_\mq\phi_k$ networks.

\begin{lemma}\label{lem:inverse_phi_k}
    Let $d,n\geq2$, $k\in\NN_0$, $\mq\in\NN_0$, $0\leq s<r<k+\frac12$, and let $\{\theta_j^*\}_{j=1}^n\subset\SS^d$ have separation distance $\tilde h:=\min_{i\neq j}\rho(\theta_i^*,\theta_j^*)>0$.
    If $2^\mq\tilde h\leq c_*$, then every $g_n\in\mathrm{span}\{\phi_k(\theta_j^*\cdot\circ)\}_{j=1}^n$ satisfies
    \begin{equation}
        \|\pc_\mq g_n\|_{\mathcal H^r(\SS^d)}\lesssim \tilde{h}^{-(r-s)}\|\pc_\mq g_n\|_{\mathcal{H}^s(\SS^d)},
    \end{equation}
\end{lemma}

\begin{proof}
For $q\geq\tilde q$, using \eqref{eqn:dig_dom_matrix}, we have
\begin{equation}
    \begin{split}
        &\sul_{q=\tilde q}^{\infty}\tilde{a}^\top Q_\rr(q)\tilde{a}
        \leq\sul_{q=\tilde q}^{\infty}\tilde{a}^\top \left(\frac{3}{2}\left(Q_{\rr}(q)\right)_{i,i}\right)I_{n\times n}\tilde{a}
        \eqsim\sul_{q=\tilde q}^{\infty}2^{-q(2k+1-2\rr)}\|\tilde{a}\|_2^2\\
        \eqsim&\tilde{h}^{2k+1-2r}\|\tilde{a}\|_2^2
        \eqsim\tilde{h}^{-2(\rr-\ss)}\sul_{q=\tilde q}^{\infty}2^{-q(2k+1-2\ss)}\|\tilde{a}\|_2^2
        \\
        \eqsim&\tilde{h}^{-2(\rr-\ss)}\sul_{q=\tilde q}^{\infty}\tilde{a}^\top \left(\frac{1}{2}\left(Q_{\ss}(q)\right)_{i,i}\right)I_{n\times n}\tilde{a}
        \leq\tilde{h}^{-2(\rr-\ss)}\sul_{q=\tilde q}^{\infty}\tilde{a}^\top Q_\ss(q)\tilde{a}.
    \end{split}
\end{equation}
On the other hand, for $\mq+1\leq q<\tilde q$, since each $P(m)$ is positive semidefinite (see, e.g., \cite[Lemma 3.4]{liu2025integral}), we conclude
\begin{equation}
    \begin{split}
        &\sul_{q=\mq+1}^{\tilde q-1}\tilde{a}^\top Q_\rr(q)\tilde{a}
        =\sul_{q=\mq+1}^{\tilde q-1}\sul_{m=2^{q-1}}^{2^{q+1}}\varphi\big(\frac{m}{2^q}\big)\xi_\rr(m)\Big(\tilde{a}^\top P(m)\tilde{a}\Big)\\
        \leq&\mal_{M\leq C\tilde h^{-1}}\Big(\frac{\xi_\rr(M)}{\xi_\ss(M)}\Big)\sul_{q=\mq+1}^{\tilde q-1}\sul_{m=2^{q-1}}^{2^{q+1}}\varphi\big(\frac{m}{2^q}\big)\xi_\ss(m)\Big(\tilde{a}^\top P(m)\tilde{a}\Big)\\
        =&\mal_{M\leq C\tilde h^{-1}}\Big(\frac{\xi_\rr(M)}{\xi_\ss(M)}\Big)\sul_{q=\mq+1}^{\tilde q-1}\tilde{a}^\top Q_\ss(q)\tilde{a}
        \eqsim\tilde{h}^{-2(\rr-\ss)}\sul_{q=\mq+1}^{\tilde q-1}\tilde{a}^\top Q_\ss(q)\tilde{a}.
    \end{split}
\end{equation}
Combining the two ranges of $q$ and recalling \eqref{eqn_sob_norm_s_r}, we conclude
\begin{equation}
    \begin{split}
        |\pc_\mq g_n|_{\mathcal{H}^\rr(\SS^d)}^2
        =&\sul_{q=\mq+1}^{\tilde q-1}\tilde{a}^\top Q_\rr(q)\tilde{a}
        +\sul_{q=\tilde q}^{\infty}\tilde{a}^\top Q_\rr(q)\tilde{a}\\
        \lesssim&\tilde{h}^{-2(\rr-\ss)}\sul_{q=\mq+1}^{\tilde q-1}\tilde{a}^\top Q_\ss(q)\tilde{a}
        +\tilde{h}^{-2(\rr-\ss)}\sul_{q=\tilde q}^{\infty}\tilde{a}^\top Q_\ss(q)\tilde{a}\\
        =&\tilde{h}^{-2(\rr-\ss)}|\pc_\mq g_n|_{\mathcal{H}^s(\SS^d)}^2.
    \end{split}
\end{equation}
Since $\mathcal P_\mq^c$ annihilates the constant mode, its spectrum is contained in degrees $m\geq1$. Hence the homogeneous seminorm and the full $\mathcal H^a$-norm are equivalent on its range for every $a>0$, while for $a=0$ the full norm is the $\mathcal L^2$-norm. The preceding seminorm estimate therefore gives the stated full-norm inequality.
\end{proof}

\subsection{Proof of Theorem \ref{thm_inverse_hig_freq}}

In this subsection, we conclude the argument by deriving Theorem~\ref{thm_inverse_hig_freq} from Lemma~\ref{lem:inverse_phi_k}.
\begin{proof}
Write
\[
f_n=\sul_{j=1}^na_j\sigma_k(\theta_j^*\cdot\circ),\qquad a=(a_1,\dots,a_n)^\top.
\]
Define
\begin{equation}
    \tilde{f}=\frac{1}{2}\left\{\begin{array}{ll}
        \sul_{j=1}^na_j\left(\phi_k(\theta_j^*\cdot\circ)-\phi_k(-\theta_j^*\cdot\circ)\right), &\quad\hbox{if $k$ is even},  \\
        \sul_{j=1}^na_j\left(\phi_k(\theta_j^*\cdot\circ)+\phi_k(-\theta_j^*\cdot\circ)\right), &\quad\hbox{if $k$ is odd}.
    \end{array}\right.
\end{equation}
and apply Lemma \ref{lem:inverse_phi_k} to $\tilde f\in\mathrm{span}\left\{\phi_k(\vartheta_j^*\cdot\circ),~j=1,\dots,2n\right\}$, where $\{\vartheta_j^*\}_{j=1}^{2n}=\{\theta_j^*\}_{j=1}^n\cup\{-\theta_j^*\}_{j=1}^n$. Then
\begin{equation}\label{eqn:tilde_f_inverse}
    \big\|\pc_\mq\tilde f\big\|_{\mH}\lesssim \underline{h}^{-(r-s)}\big\|\pc_\mq\tilde f\big\|_{\mathcal{H}^s(\SS^d)},
\end{equation}
where
$$\underline{h}=\min\limits_{i\neq j}\rho(\vartheta_i^*,\vartheta_j^*)=\min\limits_{i\neq j}\min\left\{\rho(\theta_i^*,\theta_j^*),\rho(\theta_i^*,-\theta_j^*)\right\}.$$

Next we separate the even/odd modes of the ultraspherical expansion.
Since $\{p_{2m}\}_{m=0}^\infty$ are even and $\{p_{2m+1}\}_{m=0}^\infty$ are odd, if $k$ is even, we have
\begin{equation*}
    \begin{split}
        \pc_\mq\tilde{f}
        &=\frac{1}{2}\sul_{j=1}^na_j\sul_{m=0}^\infty\widehat{\pc_\mq\phi_k}(m)\left(p_m(\theta_j^*\cdot\circ)-p_m(-\theta_j^*\cdot\circ)\right)\\
        &=\sul_{j=1}^na_j\sul_{m=0}^\infty\widehat{\pc_\mq\phi_k}(2m+1)p_{2m+1}(\theta_j^*\cdot\circ)\\
        &=\sul_{j=1}^na_j\left(\sul_{m=0}^\infty\left|\widehat{\pc_\mq\sigma_k}(m)\right|p_m(\theta_j^*\cdot\circ)-\sul_{2m\leq k}\left|\widehat{\pc_\mq\sigma_k}(2m)\right|p_{2m}(\theta_j^*\cdot\circ)\right),
    \end{split}
\end{equation*}
and if $k$ is odd,
\begin{equation*}
    \begin{split}
        \pc_\mq\tilde{f}
        &=\frac{1}{2}\sul_{j=1}^na_j\sul_{m=0}^\infty\widehat{\pc_\mq\phi_k}(m)\left(p_m(\theta_j^*\cdot\circ)+p_m(-\theta_j^*\cdot\circ)\right)\\
        &=\sul_{j=1}^na_j\sul_{m=0}^\infty\widehat{\pc_\mq\phi_k}(2m)p_{2m}(\theta_j^*\cdot\circ)\\
        &=\sul_{j=1}^na_j\left(\sul_{m=0}^\infty\left|\widehat{\pc_\mq\sigma_k}(m)\right|p_m(\theta_j^*\cdot\circ)-\sul_{2m+1\leq k}\left|\widehat{\pc_\mq\sigma_k}(2m+1)\right|p_{2m+1}(\theta_j^*\cdot\circ)\right).
    \end{split}
\end{equation*}

Introduce the low-degree correction
\begin{equation}
    p=\left\{\begin{array}{ll}
        \sul_{j=1}^n a_j \sul_{2m\leq k}\left|\widehat{\pc_\mq\sigma_k}(2m)\right|\,p_{2m}(\theta_j^*\cdot\circ), &\hbox{if $k$ is even},  \\
        \sul_{j=1}^n a_j \sul_{2m+1\leq k}\left|\widehat{\pc_\mq\sigma_k}(2m+1)\right|\,p_{2m+1}(\theta_j^*\cdot\circ), &\hbox{if $k$ is odd}.
    \end{array}\right.
\end{equation}
By parity, $\pc_\mq\tilde f\perp p$, hence
\begin{equation*}
    \begin{split}
        \big\|\pc_\mq\tilde f+p\big\|_{\mH}^2
        =&\big\|\pc_\mq\tilde f\big\|_{\mH}^2+\left\|p\right\|_{\mH}^2
        \lesssim \underline{h}^{2s-2r}\|\pc_\mq\tilde f\|_{\mathcal{H}^s(\SS^d)}^2+\left\|p\right\|_{\mH}^2\\
        \lesssim&\underline{h}^{2s-2r}\left(\|\pc_\mq\tilde f\|_{\mathcal{H}^s(\SS^d)}^2+\left\|p\right\|_{\mathcal{H}^s(\SS^d)}^2\right)
        =\underline{h}^{2s-2r}\big\|\pc_\mq\tilde f+p\big\|_{\mathcal{H}^s(\SS^d)}^2.
    \end{split}
\end{equation*}
The second inequality uses $p\in\mathbb P_k(\SS^d)$, so $\|p\|_{\mH}\eqsim\|p\|_{\mathcal H^s(\SS^d)}$.
Finally, for $\alpha=r,s$,
\begin{equation*}
    \begin{split}
        \big\|\pc_\mq\tilde f+p\big\|_{\mathcal{H}^\alpha(\SS^d)}^2
        =&\Big\|\sul_{j=1}^na_j\sul_{m=0}^\infty\big|\widehat{\pc_\mq\sigma_k}(m)\big|p_m(\theta_j^*\cdot\circ)\Big\|_{\mathcal{H}^\alpha(\SS^d)}^2\\
        =&\sul_{m=0}^\infty\widehat{\pc_\mq\sigma_k}(m)^2(m^{2\alpha}+1)\,
        a^\top\left(p_m(\theta_i^*\cdot \theta_j^*)\right)_{i,j=1}^na
        =\|\pc_\mq f_n\|_{\mathcal{H}^\alpha(\SS^d)}^2.
    \end{split}
\end{equation*}
This identifies the high-frequency norms of $\pc_\mq\sigma_k$–type networks with those of $\pc_\mq\phi_k$–type networks, and the asserted Bernstein estimate \eqref{eqn_inver_highfre} follows from \eqref{eqn:tilde_f_inverse}.

Finally, write $v_n=\mathfrak L_\beta f_n$ with $f_n\in L_n^k$. Since $\mathfrak L_\beta$ commutes with $\mathcal P_{\mathfrak q}^c$, \eqref{eqn:operator_equivalence} yields \eqref{eqn:filtered_residual_bernstein}.

\end{proof}

The Bernstein and residual Bernstein estimates established above provide the stability mechanism for the discrete least-squares analysis in the next section.

\section{Discrete residual least squares on the sphere}\label{sec_mainresult_ls}

Building on the stability estimates of Section~\ref{sec_bernstein}, we now establish deterministic and random error bounds for a residual minimizer $u_{n,m}$ in \eqref{eqn_def_collocation}. Throughout this section, $h$ and $\underline h$ denote the fill distance and antipodal separation distance of $\Theta_n$, as in Definition~\ref{def:quasiuniform_sequence}, with the index $n$ suppressed.

We assume throughout this section that the parity condition \eqref{eqn:f_evenodd_relu} holds. This is the parity associated with the nonzero high-degree spherical harmonic coefficients of $\sigma_k$; see~\cite[Remark~4.1]{liu2025integral}. Since $\mathfrak L_\beta$ preserves spherical harmonic degree, the exact solution $u=\mathfrak L_\beta^{-1}f$ has the same parity whenever $f$ satisfies \eqref{eqn:f_evenodd_relu}. A more general target of the form $f=f_{\mathrm{par}}+p$, where $f_{\mathrm{par}}$ satisfies \eqref{eqn:f_evenodd_relu} and $p\in\mathbb P_k(\mathbb S^d)$, may be treated by handling $p$ separately. We do not claim that an arbitrary function admits such a decomposition with $p\in\mathbb P_k(\mathbb S^d)$.

We first state the deterministic result, which gives the approximation rate for quasi-uniform collocation sets.

\begin{theorem}\label{thm:leastsquare}
Let $d,n\geq2$, $k\in\NN_0$, $\beta\geq0$, $k> \beta+\frac{d-1}{2}$, $r\leq\frac{d+2k+1}{2}-\beta$, $0\leq s<\min\{r,k+\frac{1}{2}-\beta\}$, $\Theta_n=\{\theta_{j,n}^*\}_{j=1}^n\subset\SS^d$, and let $\mathfrak L_\beta$ satisfy \eqref{eqn:def_operator}--\eqref{eqn:elliptic_symbol}. Then there exists a constant $C_2$ such that, whenever $m\geq C_2\underline h^{-d}$, $X=\{\eta_i^*\}_{i=1}^m$ is quasi-uniform, and $f$ satisfies \eqref{eqn:f_evenodd_relu}, the solution $u_{n,m}$ of \eqref{eqn_def_collocation} satisfies
\begin{equation}\label{eqn_main_ls_determ}
\|u-u_{n,m}\|_{\mathcal H^{s+\beta}(\SS^d)}\eqsim\|f-\mathfrak L_\beta u_{n,m}\|_{\mathcal H^s(\SS^d)}\lesssim\underline h^{-s}h^r
\begin{cases}
\|f\|_{\mathcal W^{r,p}(\SS^d)},&\frac dp<r\leq \frac{d}{2},\\
\|f\|_{\mathcal H^r(\SS^d)},&r>\frac{d}{2}.
\end{cases}
\end{equation}
The implied constant is independent of $n,m$, $\Theta_n$, and $X$. If $\{\Theta_n\}_{n\geq1}$ is antipodally quasi-uniform and $m=\lceil C_2\underline h^{-d}\rceil$, then
\begin{equation}
\|u-u_{n,m}\|_{\mathcal H^{s+\beta}(\SS^d)}\eqsim\|f-\mathfrak L_\beta u_{n,m}\|_{\mathcal H^s(\SS^d)}\lesssim n^{-\frac{r-s}{d}}
\begin{cases}
\|f\|_{\mathcal W^{r,p}(\SS^d)},&\frac dp<r\leq \frac{d}{2},\\
\|f\|_{\mathcal H^r(\SS^d)},&r>\frac{d}{2}.
\end{cases}
\end{equation}
\end{theorem}

The pointwise empirical loss creates a regularity distinction that is absent from the continuous approximation theorem. When $r\leq \frac{d}{2}$, we assume $f\in\mathcal W^{r,p}$ in order to control the sampled high-frequency tail; for $r>\frac{d}{2}$, the Hilbert Sobolev assumption $f\in\mathcal H^r$ suffices.

\begin{remark}
The pure fractional Laplace--Beltrami operator $(-\Delta)^{\beta/2}$, $\beta>0$, has the constants as its kernel and therefore does not satisfy \eqref{eqn:elliptic_symbol} at degree $m=0$. Nevertheless, the conclusion of Theorem~\ref{thm:leastsquare} remains valid on the mean-zero subspace, provided that $\int_{\mathbb S^d}f=0$, the exact solution is normalized by $\int_{\mathbb S^d}u=0$, and the discrete least-squares problem is solved under the constraint $\int_{\mathbb S^d}u_{n,m}=0$. Without the discrete mean-zero constraint, the solution error can only be controlled modulo constants:
\begin{equation*}
\inf_{c\in\mathbb R}\|u-u_{n,m}-c\|_{\mathcal H^{s+\beta}(\mathbb S^d)}
\lesssim\|f-(-\Delta)^{\beta/2}u_{n,m}\|_{\mathcal H^s(\mathbb S^d)}.
\end{equation*}
\end{remark}

We next state the corresponding near-optimal result for collocation points sampled independently from the uniform distribution on $\SS^d$.

\begin{theorem}\label{thm:leastsquare_random}
Let $d,n\geq2$, $k\in\NN_0$, $\beta\geq0$, $k>\beta+\frac{d-1}{2}$, and let $\Theta=\{\Theta_n\}_{n\geq1}$ be an antipodally quasi-uniform sequence. Assume that $r+\beta\leq\frac{d+2k+1}{2}$, and $f\in\mathcal H^r(\SS^d)$ satisfies \eqref{eqn:f_evenodd_relu}. For
\begin{equation*}
0<\epsilon<\min\left\{r,\frac d4,k+\frac12-\beta\right\},\qquad0<\delta<\frac12,
\end{equation*}
there exists $C_3>0$ such that, whenever
\begin{equation*}
m\geq C_3\max\left\{n\log\frac n\delta,\delta^{-\frac{d}{2\epsilon}}\right\},
\end{equation*}
the i.i.d.\ uniform samples $X=\{\eta_i^*\}_{i=1}^m$ yield, with probability at least $1-2\delta$,
\begin{equation}\label{eqn:random_residual_result}
\|u-u_{n,m}\|_{\mathcal H^\beta(\SS^d)}\eqsim\|f-\mathfrak L_\beta u_{n,m}\|_{\mathcal L^2(\SS^d)}\lesssim\sqrt{\log\frac m\delta}\,n^{-\frac{r-\epsilon}{d}}\|f\|_{\mathcal H^r(\SS^d)}.
\end{equation}
\end{theorem}

The remainder of this section proves Theorems~\ref{thm:leastsquare} and \ref{thm:leastsquare_random}. The deterministic argument adapts discrete least-squares ideas for polynomial spaces \cite{cohen2017optimal} to the linearized ReLU$^k$ spaces on $\SS^d$. Its main ingredients are spherical sampling inequalities, the residual Bernstein estimates of Section~\ref{sec_bernstein}, and a Marcinkiewicz--Zygmund inequality for spherical polynomials. The random argument follows the same comparison-and-stability strategy, replacing deterministic norm control by a random covering estimate and an empirical moment bound \cite{wainwright2019high,vershynin2018high,mhaskar2001spherical}.

\subsection{Auxiliary results}
For brevity, write $X=\{\eta_i^*\}_{i=1}^m$ and define
\begin{equation}
    \|g\|_{\ell^2(X)}:=\Big(\frac{1}{m}\sul_{i=1}^mg(\eta_i^*)^2\Big)^{\frac{1}{2}},\qquad g\in\mathcal{C}(\SS^d).
\end{equation}

The proofs use a comparison-network theorem together with polynomial approximation, sampling, and empirical-norm estimates. We collect these results here for convenient reference.

We first recall the approximation result from~\cite{liu2025integral} that supplies the comparison network used below. The cited theorem is stated in a weaker form than the one recorded here; the additional low-degree identity and explicit construction of $u_n$ follows from \cite[(4.12)--(4.13)]{liu2025integral}.

\begin{theorem}\label{thm:main_sphere}
Let $d,n\in\NN$, $k\in\NN_0$, $r\leq\frac{d+2k+1}{2}$, and $\{\theta_j^*\}_{j=1}^n\subset\SS^d$. For any $u\in\mathcal H^r(\SS^d)$ satisfying the parity in \eqref{eqn:f_evenodd_relu}, there exists $u_n\in L_n^k$ such that
\begin{equation}
\|u-u_n\|_{\mathcal H^s(\SS^d)}\lesssim h^{r-s}\|u\|_{\mathcal H^r(\SS^d)},
\end{equation}
holds for $0\leq s<\min\{k+\frac12,r\}$ and
\begin{equation}\label{eqn:trunc_equal}
\widehat u_n(\nu,\ell)=\widehat u(\nu,\ell),\qquad \nu\leq C_1 h^{-1},
\end{equation}
where $C_1$ is a constant independent of $n$ and $\{\theta_j^*\}_{j=1}^n$.
\end{theorem}

We next record the polynomial approximation and sampling estimates used to compare the continuous and empirical residual norms.

\begin{lemma}[Approximation by polynomials]\label{lem_sharp_appr_poly}
    Let $r>0$, $0\leq s\leq r$, $1\leq p\leq\infty$. Then for $\mq\in\NN_0$,
    \begin{equation}\label{eqn_sharp_appr_poly}
        \|u-\mP_\mq u\|_{\mathcal W^{s,p}(\SS^d)}\lesssim2^{-\mq (r-s)}\|u\|_{\mathcal{W}^{r,p}(\SS^d)}.
    \end{equation}
\end{lemma}

\begin{proof}
By combining formula (2.6.4) with Theorem 4.4.2 (when $r\in\NN$) or Theorem 4.8.1 (when $r\notin\NN$) in \cite{dai2013approximation}, we obtain the $L^p$-approximation estimate for the polynomial operator $L_n$. In our setting, the operator $\mP_\mq$ coincides with $L_n$ in \cite{dai2013approximation} when $n=2^\mq$. Therefore, applying the above results with $n=2^\mq$ yields
\begin{equation*}
    \|u-\mP_\mq u\|_{\mathcal L^p(\SS^d)}\lesssim2^{-\mq r}\|u\|_{\mathcal{W}^{r,p}(\SS^d)}.
\end{equation*}

    To prove \eqref{eqn_sharp_appr_poly}, we use the fact that the operators $(-\Delta)^{s/2}$ and $\mP_\mq$ commute. Hence
    \begin{equation*}
        \begin{split}
            \|u-\mP_\mq u\|_{\mathcal W^{s,p}(\SS^d)}
            \lesssim{}&\|u-\mP_\mq u\|_{\mathcal L^p(\SS^d)}+\|(-\Delta)^{s/2}u-\mP_\mq(-\Delta)^{s/2}u\|_{\mathcal L^p(\SS^d)}\\
            \lesssim{}&2^{-\mq r}\|u\|_{\mathcal W^{r,p}(\SS^d)}+2^{-\mq(r-s)}\|(-\Delta)^{s/2}u\|_{\mathcal W^{r-s,p}(\SS^d)}\\
            \lesssim{}&2^{-\mq(r-s)}\|u\|_{\mathcal W^{r,p}(\SS^d)},
        \end{split}
    \end{equation*}
    which completes the proof.
\end{proof}

The next result is a localized version of the classical Sobolev inequality.

\begin{lemma}\label{lem:ext_cell_sobolev}
Let $d\in\NN$, $2\leq p<\infty$, $\delta\in(0,1)$, and let $A\subset\SS^d$ be a spherical cap with radius $<\frac{\pi}{4}$. Then for continuous $v\in\mathcal W^{\frac dp+\delta,p}(A)$,
\begin{equation}\label{eqn:ext_cell_sobolev}
    \operatorname{diam}(A)^{d}|v(\eta)|^p\lesssim \|v\|_{\mathcal{L}^p(A)}^p+\operatorname{diam}(A)^{d+p\delta}\|v\|_{\mathcal W^{\frac dp+\delta,p}(A)}^p,\quad\eta\in A,
\end{equation}
where the implied constant depends only on $d$, $p$, and $\delta$.
\end{lemma}

\begin{proof}
    Without loss of generality, assume that $A$ is centered at the north pole $e_{d+1}=(0,\dots,0,1)^\top$ with radius $t$. We define the natural projection of $v$ onto the $d$-dimensional disk by
    \begin{equation*}
        v^*(x)=v(x,\sqrt{1-|x|^2}),\quad x\in\sin t\BB^d=\{(x_1,\dots,x_d)^\top:~|x|\leq\sin t\}.
    \end{equation*}
    The Sobolev norms on $A$ and $\sin t\BB^d$ are equivalent, namely
    \begin{equation}\label{eqn_equiv_cap_ball}
        \|v\|_{\mathcal{W}^{r,p}(A)}\eqsim\|v^*\|_{\mathcal{W}^{r,p}(\sin t\BB^d)},
    \end{equation}
    with constants independent of $t$. Define
    $$\tilde{v}(y)=v^*(y\sin t),\quad y\in\BB^d.$$
    Now the Sobolev embedding theorem and a scaling argument give
    \begin{equation*}
        \begin{split}
            |\tilde v(y)|^p
            \leq&\|\tilde v\|_{\mathcal{L}^{\infty}(\BB^d)}^p
            \lesssim\|\tilde v\|_{\mathcal{L}^p(\BB^d)}^p
            +|\tilde v|_{\mathcal{W}^{\frac dp+\delta,p}(\BB^d)}^p\\
            \lesssim&(\sin t)^{-d}\|v^*\|_{\mathcal{L}^p(\sin t\BB^d)}^p
            +(\sin t)^{p\delta}|v^*|_{\mathcal{W}^{\frac dp+\delta,p}(\sin t\BB^d)}^p\\
            \lesssim&\operatorname{diam}(A)^{-d}\|v\|_{\mathcal{L}^p(A)}^p
            +\operatorname{diam}(A)^{p\delta}\|v\|_{\mathcal W^{\frac dp+\delta,p}(A)}^p,\quad y\in\BB^d.
        \end{split}
    \end{equation*}
    Taking
    $$y=\frac{1}{\sin t}(\eta_1,\dots,\eta_d)^\top$$
    and multiplying both sides by $\operatorname{diam}(A)^d$ completes the proof.
\end{proof}

The following lemma is an immediate consequence of Lemma~\ref{lem:ext_cell_sobolev}; it bounds the discrete norm by continuous Sobolev norms.

The following lemma is an immediate consequence of Lemma~\ref{lem:ext_cell_sobolev}; it bounds the discrete norm by continuous Sobolev norms.

\begin{lemma}[$\mathcal L^p$ discrete norm estimate]\label{lem_discrete_by_cont}
Let $2\leq p\leq\infty$ and let $X=\{\eta_i^*\}_{i=1}^m\subset\SS^d$ be quasi-uniform with fill distance $h_X$. If
\begin{equation*}
    \frac dp<\tau<\frac dp+1,
\end{equation*}
then every $v\in\mathcal W^{\tau,p}(\SS^d)$ satisfies
\begin{equation}\label{eqn:norming}
    \|v\|_{\ell^2(X)}^2\lesssim\|v\|_{\mathcal L^p(\SS^d)}^2+h_X^{2\tau}\|v\|_{\mathcal W^{\tau,p}(\SS^d)}^2.
\end{equation}
The implied constant depends only on $d$, $p$, and $\tau$.
\end{lemma}

\begin{proof}
Let $\{A_i\}_{i=1}^m$ be the Voronoi partition associated with $X$. For each $i$, let $A_i^*$ be the spherical cap of smallest radius centered at $\eta_i^*$ that contains $A_i$. Then
\begin{equation*}
    \operatorname{diam}(A_i^*)\eqsim m^{-\frac1d}\eqsim h_X.
\end{equation*}
Writing $\tau=\frac dp+\delta$ with $\delta\in(0,1)$ and applying Lemma~\ref{lem:ext_cell_sobolev}, we obtain
\begin{equation*}
    \frac{1}{m}|v(\eta_i^*)|^p
    \lesssim \|v\|_{\mathcal{L}^p(A_i^*)}^p
    +m^{-\frac{p\tau}{d}}\|v\|_{\mathcal W^{\tau,p}(A_i^*)}^p.
\end{equation*}
Consequently,
\begin{equation*}
    \frac{1}{m}\sul_{i=1}^m|v(\eta_i^*)|^p
    \lesssim\sul_{i=1}^m\|v\|_{\mathcal{L}^p(A_i^*)}^p
    +m^{-\frac{p\tau}{d}}\sul_{i=1}^m\|v\|_{\mathcal W^{\tau,p}(A_i^*)}^p.
\end{equation*}
Each point $\eta\in\SS^d$ belongs to only a uniformly bounded number of caps in the collection $\{A_i^*\}_{i=1}^m$. Therefore,
\begin{equation*}
    \sul_{i=1}^m\|v\|_{\mathcal{L}^p(A_i^*)}^p\lesssim\|v\|_{\mathcal L^p(\SS^d)}^p,\qquad
    \sul_{i=1}^m\|v\|_{\mathcal W^{\tau,p}(A_i^*)}^p\lesssim\|v\|_{\mathcal W^{\tau,p}(\SS^d)}^p.
\end{equation*}
It follows that
\begin{equation*}
    \left(\frac{1}{m}\sul_{i=1}^m|v(\eta_i^*)|^p\right)^{1/p}
    \lesssim\|v\|_{\mathcal L^p(\SS^d)}
    +m^{-\frac{\tau}{d}}\|v\|_{\mathcal W^{\tau,p}(\SS^d)}.
\end{equation*}
Since $p\geq2$, the monotonicity of the normalized discrete norms gives
\begin{equation*}
    \|v\|_{\ell^2(X)}
    \leq\left(\frac{1}{m}\sul_{i=1}^m|v(\eta_i^*)|^p\right)^{1/p}.
\end{equation*}
Using $m^{-1/d}\eqsim h_X$ and squaring the resulting estimate completes the proof.
\end{proof}

The preceding lemma controls the empirical norm by continuous Sobolev norms. The converse estimate needed for stability is the following sampling inequality.

\begin{lemma}[Sampling inequality]\label{lem_cont_by_discrete}
Let $X=\{\eta_i^*\}_{i=1}^m\subset\SS^d$ have fill distance $h_X$, and let $\frac{d}{2}<\tau<\frac{d}{2}+1$. There exists $h_0>0$ such that, whenever $h_X\leq h_0$,
\begin{equation}\label{eqn:sampling_inequality}
\|v\|_{\mathcal L^2(\SS^d)}^2\lesssim h_X^d\sul_{i=1}^m|v(\eta_i^*)|^2+h_X^{2\tau}|v|_{\mathcal H^\tau(\SS^d)}^2,\qquad v\in\mathcal H^\tau(\SS^d).
\end{equation}
\end{lemma}

\begin{proof}
Apply the Euclidean sampling inequality of \cite{arcangeli2014sampling} in a fixed finite atlas and combine the chartwise estimates with a smooth partition of unity. The chart argument first produces the full $\mathcal H^\tau$-norm in the remainder. Since $\|v\|_{\mathcal H^\tau}^2=\|v\|_{\mathcal L^2}^2+|v|_{\mathcal H^\tau}^2$, the additional term $h_X^{2\tau}\|v\|_{\mathcal L^2}^2$ is absorbed into the left-hand side by decreasing $h_0$.
\end{proof}

For completeness, we also record a classical Marcinkiewicz--Zygmund inequality on the sphere; see \cite[Theorem 3.1]{mhaskar2001spherical}.

\begin{lemma}[Marcinkiewicz–Zygmund inequality for polynomials]\label{lem_MZ_L2}
Let $X=\{\eta_i^*\}_{i=1}^m$ be quasi-uniform and fill distance $h_X$. For a sufficiently small constant $c_2>0$ and every polynomial $v$ with $\deg(v)\leq c_2h_X^{-1}$,
    \begin{equation}
        \|v\|_{\ell^2(X)}\eqsim\|v\|_{\mathcal{L}^2(\SS^d)}.
    \end{equation}
\end{lemma}
The cited theorem is naturally stated with local cell weights. For a quasi-uniform set, the cell measures are comparable to $m^{-1}$, which gives the equal-weight form above.

For random collocation points, we use the following empirical moment estimate.

\begin{lemma}[Marcinkiewicz–Zygmund inequality for random variables]\label{lem_rand_l2Lp}
    Let $g\in\mathcal L^p(\SS^d)$ be fixed, and let $X = \{\eta_i^*\}_{i=1}^m$ consist of i.i.d.\ samples drawn from the uniform distribution on $\mathbb{S}^d$. Then for any $p\in(2,4)$, with probability at least $1-m^{2/p-1}$,
    \begin{equation}
        \|g\|_{\ell^2(X)}\leq\sqrt{C_p+1}\,\|g\|_{\mathcal{L}^p(\SS^d)},
    \end{equation}
    where $C_p$ depends only on $p$.
\end{lemma}

\begin{proof}
Applying the classical Marcinkiewicz--Zygmund moment inequality for sums of independent random variables (see, e.g., \cite{vershynin2018high}), we have
\begin{equation}
    \begin{split}
    \mathbb E\Big[\big|\|g\|_{\ell^2(X)}^2-\|g\|_{\mathcal{L}^2(\SS^d)}^2\big|\Big]
    &\leq C_pm^{2/p-1}\|g^2\|_{\mathcal{L}^{p/2}(\SS^d)}
    =C_pm^{2/p-1}\|g\|_{\mathcal{L}^{p}(\SS^d)}^2,
    \end{split}
\end{equation}
where we use the Marcinkiewicz--Zygmund, equivalently von Bahr--Esseen, moment bound in the range $p/2\in(1,2)$. Markov's inequality gives
\begin{equation*}
\mathrm{Prob}\left(\big|\|g\|_{\ell^2(X)}^2-\|g\|_{\mathcal L^2(\SS^d)}^2\big|>C_p\|g\|_{\mathcal L^p(\SS^d)}^2\right)\leq m^{2/p-1}.
\end{equation*}
On the complementary event, $\|g\|_{\mathcal L^2}\leq\|g\|_{\mathcal L^p}$, up to the same fixed normalization, and therefore
\begin{equation}
    \|g\|_{\ell^2(X)}^2\leq(C_p+1)\|g\|_{\mathcal{L}^{p}(\SS^d)}^2.
\end{equation}
\end{proof}

\subsection{Proof of Theorem~\ref{thm:leastsquare}}
\begin{proof}
Let $u=\mathfrak L_\beta^{-1}f$, and let $u_n$ be the function in Theorem~\ref{thm:main_sphere} obtained by applying that theorem to $u$ with smoothness parameters $r+\beta$ and $s+\beta$. Since $\mathfrak L_\beta$ is an elliptic spectral multiplier of order $\beta$,
\begin{equation}\label{eqn_net_appr_u}
    \|\mathfrak L_\beta u_n-f\|_{\mathcal H^s(\SS^d)}
    \eqsim\|u_n-u\|_{\mathcal H^{s+\beta}(\SS^d)}
    \lesssim h^{r-s}\|u\|_{\mathcal H^{r+\beta}(\SS^d)}
    \eqsim h^{r-s}\|f\|_{\mathcal H^r(\SS^d)}.
\end{equation}
Here $\mathfrak L_\beta^{-1}$ preserves spherical harmonic degrees and hence preserves the parity condition required in Theorem~\ref{thm:main_sphere}. Let
\begin{equation*}
    g=\mathfrak L_\beta(u_{n,m}-u_n)\in V_{n,\beta}.
\end{equation*}
We first estimate $\|g\|_{\mathcal L^2(\SS^d)}$. Choose
$$
\frac d2<\tau<\min\big\{k+\frac12-\beta,\frac{d}{2}+1\big\},
$$
and set $\mathfrak p=\lfloor\log_2(c_2m^{1/d})\rfloor-1$, where $c_2$ is the constant in Lemma~\ref{lem_MZ_L2}. Then
\begin{equation}\label{eqn_samp_step1}
    \begin{split}
        \|g\|_{\mathcal L^2(\SS^d)}
        \leq&\|\mathcal P_{\mathfrak p}g\|_{\mathcal L^2(\SS^d)}
        +\|g-\mathcal P_{\mathfrak p}g\|_{\mathcal L^2(\SS^d)}\\
        \lesssim&\|\mathcal P_{\mathfrak p}g\|_{\ell^2(X)}
        +m^{-\frac{\tau}{d}}\|g\|_{\mathcal H^\tau(\SS^d)},
    \end{split}
\end{equation}
where the last inequality follows by applying Lemma~\ref{lem_MZ_L2} to $\mathcal P_{\mathfrak p}g$ and Lemma~\ref{lem_sharp_appr_poly} to $g$.

Applying Lemma~\ref{lem_discrete_by_cont} to $\pc_{\mathfrak p}g$ gives
\begin{equation}\label{eqn_est_Pc_ell2}
    \begin{split}
        \|\pc_{\mathfrak p}g\|_{\ell^2(X)}^2
        \lesssim&\|\pc_{\mathfrak p}g\|_{\mathcal L^2(\SS^d)}^2
        +m^{-\frac{2\tau}{d}}\|\pc_{\mathfrak p}g\|_{\mathcal H^\tau(\SS^d)}^2\\
        \lesssim&\left(2^{-2\mathfrak p\tau}+m^{-\frac{2\tau}{d}}\right)
        \|g\|_{\mathcal H^\tau(\SS^d)}^2
        \lesssim m^{-\frac{2\tau}{d}}\|g\|_{\mathcal H^\tau(\SS^d)}^2,
    \end{split}
\end{equation}
where the second inequality follows from Lemma~\ref{lem_sharp_appr_poly} and the boundedness of $\pc_{\mathfrak p}$ on $\mathcal H^\tau(\SS^d)$. Combining \eqref{eqn_samp_step1} and \eqref{eqn_est_Pc_ell2}, we obtain
\begin{equation}\label{eqn:L2_by_discrete}
    \begin{split}
        \|g\|_{\mathcal L^2(\SS^d)}
        \lesssim&\left(\|\pc_{\mathfrak p}g\|_{\ell^2(X)}+\|g\|_{\ell^2(X)}\right)
        +m^{-\frac{\tau}{d}}\|g\|_{\mathcal H^\tau(\SS^d)}\\
        \lesssim&\|g\|_{\ell^2(X)}+m^{-\frac{\tau}{d}}\|g\|_{\mathcal H^\tau(\SS^d)}.
    \end{split}
\end{equation}

\medskip
Since $g\in V_{n,\beta}$, the residual Bernstein estimate \eqref{eqn:residual_bernstein} from Theorem~\ref{thm:inverse} gives
\begin{equation}\label{eqn:residual_bernstein_used}
    \|g\|_{\mathcal H^\tau(\SS^d)}
    \lesssim\underline h^{-\tau}\|g\|_{\mathcal L^2(\SS^d)},
    \qquad
    \|g\|_{\mathcal H^s(\SS^d)}
    \lesssim\underline h^{-s}\|g\|_{\mathcal L^2(\SS^d)}.
\end{equation}
For $s=0$, the second inequality is understood trivially. It follows from \eqref{eqn:L2_by_discrete} that
\begin{equation*}
    \|g\|_{\mathcal L^2(\SS^d)}
    \lesssim\|g\|_{\ell^2(X)}
    +(m\underline h^d)^{-\frac{\tau}{d}}\|g\|_{\mathcal L^2(\SS^d)}.
\end{equation*}
Hence, there exists a constant $C_2>0$ such that, whenever $m\geq C_2\underline h^{-d}$, the last term can be absorbed into the left-hand side. Therefore,
\begin{equation}\label{eqn:residual_stability}
    \|g\|_{\mathcal H^s(\SS^d)}
    \lesssim\underline h^{-s}\|g\|_{\mathcal L^2(\SS^d)}
    \lesssim\underline h^{-s}\|g\|_{\ell^2(X)}.
\end{equation}
Since $u_{n,m}$ is the $\ell^2$-minimizer of $f-\mathfrak L_\beta v_n$ over $L_n^k$, we obtain
\begin{equation}\label{eqn:error_decom2}
    \begin{split}
        \|g\|_{\mathcal H^s(\SS^d)}
        \lesssim&\underline h^{-s}\|\mathfrak L_\beta(u_{n,m}-u_n)\|_{\ell^2(X)}\\
        \lesssim&\underline h^{-s}\|\mathfrak L_\beta u_{n,m}-f\|_{\ell^2(X)}
        +\underline h^{-s}\|f-\mathfrak L_\beta u_n\|_{\ell^2(X)}\\
        \leq&2\underline h^{-s}\|f-\mathfrak L_\beta u_n\|_{\ell^2(X)}.
    \end{split}
\end{equation}

\medskip
We now estimate $\|f-\mathfrak L_\beta u_n\|_{\ell^2(X)}$. Let $C_1$ be the constant in Theorem \ref{thm:main_sphere} and $c_*$ be the constant in Theorem \ref{thm_inverse_hig_freq}. Choose a fixed integer $q_0\geq1$ sufficiently large and set
\begin{equation*}
    \mq=\lfloor\log_2C_1 h^{-1}\rfloor-q_0,
\end{equation*}
so that
\begin{equation*}
    2^{\mq+1}\leq C_1 h^{-1},\qquad 2^\mq\underline h\leq c_*,\qquad 2^{-\mq}\eqsim h.
\end{equation*}
The finitely many coarse cases for which $\mq<0$ can be included by enlarging the implied constant; hence, in the remainder of the proof, we assume $\mq\geq0$. By Theorem~\ref{thm:main_sphere}, $\mathcal P_\mq u=\mathcal P_\mq u_n$. Since $\mathfrak L_\beta$ commutes with $\mathcal P_\mq$, it follows that
\begin{equation*}
    \mathcal P_\mq f=\mathcal P_\mq\mathfrak L_\beta u_n.
\end{equation*}
Consequently,
\begin{equation}\label{eqn:error_decom3}
    \begin{split}
        \|f-\mathfrak L_\beta u_n\|_{\ell^2(X)}
        \leq&\|f-\mathcal P_\mq f\|_{\ell^2(X)}
        +\|\mathcal P_\mq f-\mathfrak L_\beta u_n\|_{\ell^2(X)}\\
        =&\|\pc_\mq f\|_{\ell^2(X)}
        +\|\pc_\mq\mathfrak L_\beta u_n\|_{\ell^2(X)}.
    \end{split}
\end{equation}
The two terms are estimated separately. For $\|\pc_\mq f\|_{\ell^2(X)}$:
\begin{itemize}
    \item If $r\leq d/2$, let $2<p<\infty$ satisfy $r>d/p$ and choose
\begin{equation*}
    \frac{d}{p}<\tau<\min\left\{r,\frac{d}{p}+1\right\}.
\end{equation*}
Applying Lemma~\ref{lem_discrete_by_cont} and Lemma~\ref{lem_sharp_appr_poly} with $p$ yields
\begin{equation*}
    \begin{split}
        \|\pc_\mq f\|_{\ell^2(X)}
        &\lesssim\|\pc_\mq f\|_{\mathcal L^p(\SS^d)}+h_X^\tau\|\pc_\mq f\|_{\mathcal W^{\tau,p}(\SS^d)}\\
        &\lesssim\left(2^{-\mq r}+h_X^\tau2^{-\mq(r-\tau)}\right)\|f\|_{\mathcal W^{r,p}(\SS^d)}
        \lesssim h^r\|f\|_{\mathcal W^{r,p}(\SS^d)}.
    \end{split}
\end{equation*}
    \item If $r>d/2$, choose
    \begin{equation*}
        \frac d2<\tau_1<\min\left\{r,\frac d2+1\right\}.
    \end{equation*}
    Applying Lemmas~\ref{lem_discrete_by_cont} and \ref{lem_sharp_appr_poly}, and using $m^{-1/d}\lesssim\underline h\lesssim h$, we obtain
    \begin{equation*}
        \begin{split}
            \|\pc_\mq f\|_{\ell^2(X)}
            \lesssim&\|\pc_\mq f\|_{\mathcal L^2(\SS^d)}
            +m^{-\frac{\tau_1}{d}}\|\pc_\mq f\|_{\mathcal H^{\tau_1}(\SS^d)}\\
            \lesssim&\left(h^r+m^{-\frac{\tau_1}{d}}h^{r-\tau_1}\right)
            \|f\|_{\mathcal H^r(\SS^d)}
            \lesssim h^r\|f\|_{\mathcal H^r(\SS^d)}.
        \end{split}
    \end{equation*}
\end{itemize}
Therefore,
\begin{equation}\label{eqn_poincare1}
    \|\pc_\mq f\|_{\ell^2(X)}
    \lesssim h^r
    \left\{\begin{array}{ll}
        \|f\|_{\mathcal W^{r,p}(\SS^d)},&\frac{d}{p}<r\leq\frac d2,\\[0.3em]
        \|f\|_{\mathcal H^r(\SS^d)},&r>\frac d2.
    \end{array}\right.
\end{equation}

\medskip
For $\|\pc_\mq\mathfrak L_\beta u_n\|_{\ell^2(X)}$, we apply Lemma~\ref{lem_discrete_by_cont} and the filtered residual Bernstein inequality \eqref{eqn:filtered_residual_bernstein}:
\begin{equation}\label{eqn_poincare2}
    \begin{split}
        \|\pc_\mq\mathfrak L_\beta u_n\|_{\ell^2(X)}
        \lesssim&\|\pc_\mq\mathfrak L_\beta u_n\|_{\mathcal L^2(\SS^d)}
        +m^{-\frac{\tau}{d}}\|\pc_\mq\mathfrak L_\beta u_n\|_{\mathcal H^\tau(\SS^d)}\\
        \lesssim&\|\pc_\mq\mathfrak L_\beta u_n\|_{\mathcal L^2(\SS^d)}
        +m^{-\frac{\tau}{d}}\underline h^{-\tau}
        \|\pc_\mq\mathfrak L_\beta u_n\|_{\mathcal L^2(\SS^d)}\\
        \lesssim&\|\pc_\mq\mathfrak L_\beta u_n\|_{\mathcal L^2(\SS^d)}=\|\pc_\mq f+(\mathfrak L_\beta u_n-f)\|_{\mathcal L^2(\SS^d)}
        \lesssim h^r\|f\|_{\mathcal H^r(\SS^d)}.
    \end{split}
\end{equation}
Here the second inequality follows from \eqref{eqn:filtered_residual_bernstein}, the third one uses $m\geq C_2\underline h^{-d}$, and the equality in the last line follows from
\begin{equation*}
    \mathcal P_\mq(f-\mathfrak L_\beta u_n)=0.
\end{equation*}
The last inequality follows from Lemma~\ref{lem_sharp_appr_poly} and \eqref{eqn_net_appr_u} with $s=0$. In the case $r\leq d/2$, we also use
\begin{equation*}
    \|f\|_{\mathcal H^r(\SS^d)}\lesssim\|f\|_{\mathcal W^{r,p}(\SS^d)},\quad p\geq2.
\end{equation*}

\medskip
Combining \eqref{eqn:error_decom2}--\eqref{eqn_poincare2}, we obtain
\begin{equation}
    \begin{split}
        \|g\|_{\mathcal H^s(\SS^d)}
        \lesssim&\underline h^{-s}\left(\|\pc_\mq\mathfrak L_\beta u_n\|_{\ell^2(X)}
        +\|\pc_\mq f\|_{\ell^2(X)}\right)\\
        \lesssim&\underline h^{-s}h^r
        \left\{\begin{array}{ll}
            \|f\|_{\mathcal W^{r,p}(\SS^d)},&\frac{d}{2}<r\leq\frac d2,\\[0.3em]
            \|f\|_{\mathcal H^r(\SS^d)},&r>\frac d2.
        \end{array}\right.
    \end{split}
\end{equation}

Together with \eqref{eqn_net_appr_u}, this gives
\begin{equation}
    \begin{split}
        \|\mathfrak L_\beta u_{n,m}-f\|_{\mathcal H^s(\SS^d)}
        \leq&\|g\|_{\mathcal H^s(\SS^d)}
        +\|f-\mathfrak L_\beta u_n\|_{\mathcal H^s(\SS^d)}\\
        \lesssim&\underline h^{-s}h^r
        \left\{\begin{array}{ll}
            \|f\|_{\mathcal W^{r,p}(\SS^d)},&\frac dp<r\leq\frac d2,\\[0.3em]
            \|f\|_{\mathcal H^r(\SS^d)},&r>\frac d2.
        \end{array}\right.,
    \end{split}
\end{equation}
where we have used $\underline h\lesssim h$ to absorb the term $h^{r-s}$. This proves Theorem~\ref{thm:leastsquare}.
\end{proof}

\subsection{Proof of Theorem \ref{thm:leastsquare_random}}
\begin{proof}
Choose
\begin{equation*}
\frac d2<\tau<\min\left\{\frac d2+1,k+\frac12-\beta\right\}.
\end{equation*}
Let $u=\mathfrak L_\beta^{-1}f$ and let $u_n$ be the comparison network supplied by Theorem~\ref{thm:main_sphere}. With
\begin{equation*}
e_n:=f-\mathfrak L_\beta u_n,
\end{equation*}
the operator equivalence and Theorem~\ref{thm:main_sphere} give, for $t=0$ and $t=\epsilon$,
\begin{equation}\label{eqn:random_comparator}
\|e_n\|_{\mathcal H^t(\SS^d)}\lesssim h^{r-t}\|f\|_{\mathcal H^r(\SS^d)}\lesssim n^{-(r-t)/d}\|f\|_{\mathcal H^r(\SS^d)}.
\end{equation}

A standard covering argument on the sphere gives an event $\mathcal A$ of probability at least $1-\delta$ on which
\begin{equation}\label{eqn:random_fill}
h_X\lesssim\left(\frac{\log(m/\delta)}{m}\right)^{1/d}.
\end{equation}
Indeed, cover $\SS^d$ by $\mathcal O(t^{-d})$ caps of radius $t$, use the bound $\exp(-cmt^d)$ for the probability that a fixed cap is empty, and take $t^d$ to be a sufficiently large multiple of $\frac{1}{m}\log(\frac m\delta)$.

After increasing $C_3$ if necessary, the sample-size condition and \eqref{eqn:random_fill} ensure $h_X\leq h_0$ on $\mathcal A$, with $h_0$ as in Lemma~\ref{lem_cont_by_discrete}. On $\mathcal A$, set
\begin{equation*}
g_n:=\mathfrak L_\beta(u_{n,m}-u_n)\in V_{n,\beta}.
\end{equation*}
The sampling inequality and \eqref{eqn:residual_bernstein} yield
\begin{equation}
\begin{split}
\|g_n\|_{\mathcal L^2}^2
&\lesssim h_X^d\sul_{i=1}^m|g_n(\eta_i^*)|^2+h_X^{2\tau}\|g_n\|_{\mathcal H^\tau}^2\\
&\lesssim mh_X^d\|g_n\|_{\ell^2(X)}^2+\left(\frac{h_X}{\underline h}\right)^{2\tau}\|g_n\|_{\mathcal L^2}^2.
\end{split}
\end{equation}
Since $\underline h\eqsim n^{-1/d}$, \eqref{eqn:random_fill} gives
\begin{equation*}
\left(\frac{h_X}{\underline h}\right)^{2\tau}
\lesssim\left(\frac{n\log(m/\delta)}{m}\right)^{\frac{2\tau}{d}}.
\end{equation*}
The condition $m\geq C_3n\log(n/\delta)$ makes this coefficient sufficiently small to be absorbed when $C_3$ is increased. The same sample-size condition and \eqref{eqn:random_fill} give $mh_X^d\lesssim\log(m/\delta)$. Consequently,
\begin{equation}\label{eqn:random_stability}
\|g_n\|_{\mathcal L^2}\lesssim\sqrt{\log\frac m\delta}\,\|g_n\|_{\ell^2(X)}.
\end{equation}
By empirical minimality,
\begin{equation}\label{eqn:random_minimality}
\|g_n\|_{\ell^2(X)}\leq\|f-\mathfrak L_\beta u_{n,m}\|_{\ell^2(X)}+\|e_n\|_{\ell^2(X)}\leq2\|e_n\|_{\ell^2(X)}.
\end{equation}

Take $p=2d/(d-2\epsilon)$. The restriction $0<\epsilon<d/4$ ensures $2<p<4$. Lemma~\ref{lem_rand_l2Lp}, applied to the fixed function $e_n$, gives an event $\mathcal B$ of probability at least $1-m^{-2\epsilon/d}$ on which
\begin{equation}
\|e_n\|_{\ell^2(X)}\lesssim\|e_n\|_{\mathcal L^p(\SS^d)}\lesssim\|e_n\|_{\mathcal H^\epsilon(\SS^d)}\lesssim n^{-\frac{r-\epsilon}{d}}\|f\|_{\mathcal H^r(\SS^d)},
\end{equation}
where the middle inequality is the Sobolev embedding on $\SS^d$ and the last one is \eqref{eqn:random_comparator}. Since $m\geq C_3\delta^{-\frac{d}{2\epsilon}}$, one has $m^{-2\epsilon/d}\leq\delta$ after increasing $C_3$ if necessary. Thus $\mathcal A\cap\mathcal B$ has probability at least $1-2\delta$.

On this intersection, \eqref{eqn:random_comparator}, \eqref{eqn:random_stability}, and \eqref{eqn:random_minimality} imply
\begin{equation*}
\begin{split}
\|f-\mathfrak L_\beta u_{n,m}\|_{\mathcal L^2}
&\leq\|e_n\|_{\mathcal L^2}+\|g_n\|_{\mathcal L^2}\\
&\lesssim\sqrt{\log\frac m\delta}\,n^{-\frac{r-\epsilon}{d}}\|f\|_{\mathcal H^r},
\end{split}
\end{equation*}
which proves \eqref{eqn:random_residual_result}. The solution estimate follows from \eqref{eqn:operator_equivalence}.
\end{proof}

This completes the deterministic and random spherical least-squares analysis. We next restrict to the order-zero case and relate the spherical construction to affine networks on bounded domains.

\section{Affine-network interpretation on bounded domains: the order-zero case}
\label{sec_bounded_domain}

This section is restricted to the order-zero case $\beta=0$, for which $\mathfrak L_0=I$ and the equation reduces to $u=f$. The construction below is therefore an affine-network consequence of the spherical function-approximation result, rather than an extension of the positive-order PDE theory to bounded domains. For $\beta>0$, applying $\mathfrak L_\beta$ to a global spherical extension and then restricting to a cap produces an extension-dependent operator that does not, in general, coincide with a prescribed elliptic operator on $\Omega$. Moreover, the present formulation contains neither boundary residuals nor boundary conditions. To match the standard function-approximation notation used in this section and in the numerical experiments, we denote the order-zero target and its approximant by $f$ and $f_{n,m}$; here $u=f$.

We now relate the order-zero spherical formulation to the usual Euclidean formulation of shallow ReLU$^k$ networks on a bounded Lipschitz domain $\Omega\subset\RR^d$. The numerical experiments in Section~\ref{sec:numerical_experiments} show that the discrete least-squares approximation on the sphere follows the convergence behavior predicted by the $\beta=0$ specialization of Theorem~\ref{thm:leastsquare}. In contrast, experiments with direct affine ReLU$^k$ least-squares approximation on Euclidean domains suggest that a direct replacement of $\SS^d$ by $\Omega$ does not preserve the same observed behavior. This motivates the following question: how can the spherical approximation theory be related to standard affine ReLU$^k$ networks on bounded domains?

The discrete least-squares theory developed in Section~\ref{sec_mainresult_ls} is intrinsically spherical. Its proof uses the harmonic structure of $\SS^d$, including the spherical harmonic decomposition, the ultraspherical expansion of the zonal activation $\sigma_k(\theta\cdot\eta)$, and the parity condition
\[
    F(\eta)=(-1)^{k+1}F(-\eta).
\]
These structures are not available on a general bounded domain. Therefore, even the order-zero specialization of Theorem~\ref{thm:leastsquare} cannot be transferred to $\Omega$ by simply replacing $\SS^d$ with $\Omega$, and the present work does not prove stability of least squares formed from samples in $\Omega$. We record instead an indirect construction: extend the target to the sphere, apply the $\beta=0$ spherical theorem there, and restrict the induced affine network to $\Omega$.

The point of this section is to justify the following passage:
a spherical linearized network
\[
    \eta\mapsto
    \sum_{j=1}^n a_j\sigma_k(\theta_j^*\cdot\eta),
    \qquad \theta_j^*\in\SS^d,
\]
when restricted to a suitable spherical cap, induces the ordinary affine
ReLU$^k$ network
\[
    x\mapsto
    \sum_{j=1}^n a_j\sigma_k(w_j^*\cdot x+b_j^*),
    \qquad x\in\Omega,
\]
where
\[
    \theta_j^*=\binom{w_j^*}{b_j^*}\in\SS^d .
\]
Thus the least-squares approximant on $\Omega$ will be constructed by lifting
the target function to the sphere, applying the spherical discrete
least-squares theorem, and then pulling the approximant back to $\Omega$.

After a translation and dilation of the physical domain, we may assume
without loss of generality that
\[
    \overline\Omega\subset \BB^d
    :=
    \{x\in\RR^d: |x|<1\}.
\]
We identify $\BB^d$ with the spherical cap
\[
    G
    :=
    \left\{
        \eta\in\SS^d:
        \eta_{d+1}> \frac1{\sqrt2}
    \right\}.
\]
The construction consists of three steps.  First, extend the target
function from $\Omega$ to $\BB^d$.  Second, lift the extended function from
$\BB^d$ to the cap $G$ by the homogeneous change of variables associated
with ReLU$^k$.  Third, extend the lifted function from $G$ to the whole
sphere so that the parity condition required in Section~\ref{sec_mainresult_ls}
is satisfied. Once this is done, the $\beta=0$ specialization of Theorem~\ref{thm:leastsquare} applies
directly on $\SS^d$, and the resulting estimate is pulled back to
$\Omega$.

\subsection{Extension to the sphere}

Let
\[
G=\left\{\eta\in\mathbb S^d:
\eta_{d+1}>\frac1{\sqrt2}\right\}
\]
be the spherical cap corresponding to $\BB^d$.
For $x\in \BB^d$, define
\[
    \widetilde x=\binom{x}{1} .
\]
For $\eta\in G$, write
\[
    \bar\eta=(\eta_1,\ldots,\eta_d)^T .
\]

Following \cite{bach2017breaking,liu2025integral}, introduce the
homogeneous transformations
\begin{equation}\label{eq:SkTk}
\begin{aligned}
(S_k g)(x)
    &=|\widetilde x|^k
      g\left(\frac{\widetilde x}{|\widetilde x|}\right),
      \qquad x\in \BB^d,\\
(T_k f)(\eta)
    &=\eta_{d+1}^{\,k}
      f\left(\frac{\bar\eta}{\eta_{d+1}}\right),
      \qquad \eta\in G .
\end{aligned}
\end{equation}

The following properties are standard consequences of the homogeneity of
ReLU$^k$.

\begin{lemma}\label{lem:SkTk}
For $r\ge0$ and $1\le p\le\infty$,
\[
    \|T_k f\|_{\mathcal W^{r,p}(G)}
    \eqsim
    \|f\|_{\mathcal W^{r,p}(\BB^d)} ,
\]
and
\[
    S_kT_kf=f\quad\text{on }\BB^d;\qquad T_kS_kg=g\quad\text{on }G.
\]
Moreover,
\[
S_k\big(\sigma_k(\theta\cdot\eta)\big)(x)
=
\sigma_k(\theta\cdot\widetilde x).
\]
\end{lemma}

The first equivalence in Lemma~\ref{lem:SkTk} follows from the smooth diffeomorphism between the cap $G$ and the ball $\BB^d$ and standard Sobolev transformation estimates; see, for example, \cite{adams2003sobolev}. The last identity follows from the positive homogeneity of $\sigma_k$.

To apply Theorem~\ref{thm:leastsquare}, the lifted function on the sphere must satisfy the parity condition
\[
\mathcal{E}(f)(\eta)=(-1)^{k+1}\mathcal{E}(f)(-\eta).
\]
We first extend the target function from $\Omega$ to $\BB^d$. Let
\[
\mathcal E_1:\mathcal H^r(\Omega)\to\mathcal H^r(\BB^d)
\]
be a Sobolev extension operator. We recall that such an operator is constructive under standard geometric assumptions on the boundary. In particular, if $\Omega$ is a bounded Lipschitz domain, Stein's extension theorem gives a bounded linear extension operator preserving Sobolev regularity; equivalently, one may first extend to $\RR^d$ and then restrict to $\BB^d$. The construction is local: near each boundary point, one represents $\Omega$ in a boundary chart as a special Lipschitz domain, flattens the boundary, extends across the flattened boundary by a reflected averaging formula with suitable moment conditions, and finally patches the local extensions by a smooth partition of unity. Thus the extension is not an abstract existence result, although its direct numerical realization requires the boundary atlas and the corresponding partition of unity. In smoother geometries, for example when $\partial\Omega$ is $\mathcal{C}^2$, the same construction can be expressed more concretely by local normal coordinates or local reflection across the boundary. See \cite{stein1970singular,adams2003sobolev,burenkov1998sobolev}. We use the same notation for the corresponding bounded extension on the $\mathcal W^{r,p}$ scale, including $p=\infty$; in the low-regularity branch it is chosen to preserve the continuous representative assumed in Theorem~\ref{thm:leastsquareOmega}.

We next lift the extended function to the spherical cap. Let $\SS^d_+=\{\eta\in\SS^d:\eta_{d+1}>0\}$. Since $G$ has smooth boundary inside $\SS^d_+$, there exists a bounded extension operator
\[
\mathcal E_2:\mathcal H^r(G)\to\mathcal H^r(\SS^d_+).
\]
This operator can again be constructed in local coordinate charts by the same localization--reflection--patching procedure. We choose the same construction bounded on the $\mathcal W^{r,p}$ scale and preserving the continuous representative in the low-regularity branch. Choose a smooth cutoff $\chi\in \mathcal{C}^\infty(\SS^d_+)$ such that $\chi=1$ on $\overline G$ and $\chi=0$ for $0<\eta_{d+1}\le1/2$, and define $\mathcal Zg=\chi g$. Then $\mathcal Z\mathcal E_2T_k\mathcal E_1(f)$ agrees with $T_k\mathcal E_1(f)$ on $G$ and vanishes in a neighborhood of the equator.

For simplicity, we denote the above spherical extension by
\begin{equation}\label{eq:spherical_extension}
\mathcal{E}(f)(\eta)=\begin{cases}
\mathcal Z\mathcal E_2T_k\mathcal E_1(f)(\eta),&\eta_{d+1}\ge0,\\[0.3em]
(-1)^{k+1}\mathcal Z\mathcal E_2T_k\mathcal E_1(f)(-\eta),&\eta_{d+1}<0.
\end{cases}
\end{equation}
As the upper-hemisphere function vanishes in a neighborhood of the equator, the reflected definition introduces no loss of Sobolev regularity across the equator. By construction, $\mathcal{E}(f)$ satisfies
\[
\mathcal{E}(f)(\eta)=(-1)^{k+1}\mathcal{E}(f)(-\eta).
\]
Moreover, by the boundedness of $\mathcal E_1$, $\mathcal E_2$, the cutoff operator $\mathcal Z$, and the norm equivalence in Lemma~\ref{lem:SkTk},
\begin{equation}\label{eq:extension_regular}
\|\mathcal{E}(f)\|_{\mathcal H^r(\SS^d)}\lesssim\|f\|_{\mathcal H^r(\Omega)}.
\end{equation}
Similarly, for the corresponding $\mathcal{W}^{r,p}$ scale used in the statement,
\[
\|\mathcal{E}(f)\|_{\mathcal W^{r,p}(\SS^d)}\lesssim\|f\|_{\mathcal W^{r,p}(\Omega)}.
\]

\subsection{An indirect spherical-extension approximant}

We now transfer the spherical least-squares approximation result to
$\Omega$. Let $n\in\NN$ and
\[
\Theta_n=\{\theta_{j}^*\}_{j=1}^n\subset\mathbb S^d,\qquad n=1,2,\dots.
\]
For fixed $n$, we write $\theta_j^*$ instead of $\theta_{j,n}^*$.
Given spherical sampling points
$\{\eta_i^*\}_{i=1}^m\subset\mathbb S^d$,
choose $\mathcal{E}(f)_{n,m}$ and its coefficient vector $a=(a_1,\ldots,a_n)^\top$ so that
\begin{equation*}
    \mathcal{E}(f)_{n,m}\in\arg\min\limits_{g\in L_n^k(\Theta_n)}\frac{1}{m}\sum_{i=1}^m\left(
\mathcal{E}(f)(\eta_i^*)
-
g(\eta_i^*)
\right)^2,
\end{equation*}
in other words,
\begin{equation*}
a\in
\arg\min_{c\in\mathbb R^n}
\frac1m
\sum_{i=1}^m
\left(
\mathcal{E}(f)(\eta_i^*)
-
\sum_{j=1}^n
c_j\sigma_k(\theta_{j,n}^*\cdot\eta_i^*)
\right)^2 .
\end{equation*}
Now the corresponding approximant on $\Omega$ is obtained by applying the homogeneous transformation $S_k$:
\begin{equation}\label{eqn_dls_domain}
    f_{n,m}(x)=S_k(\mathcal E(f)_{n,m})(x)
=\sul_{j=1}^na_j\sigma_k(w_{j,n}^*\cdot x+b_{j,n}^*),\quad x\in\Omega.
\end{equation}

\begin{theorem}\label{thm:leastsquareOmega}
    Let $d,n\geq2$, $k\in\NN_0$, $k> \frac{d-1}{2}$, $0<r\leq\frac{d+2k+1}{2}$, $0\leq s<\min\{r,k+\frac{1}{2}\}$, $\overline\Omega\subset\BB^d$ be a Lipschitz domain, and $\Theta_n=\displaystyle\Bigg\{\binom{w_{j,n}^*}{b_{j,n}^*}\Bigg\}_{j=1}^n=\{\theta_{j,n}^*\}_{j=1}^n\subset\SS^d$. Then there exists a constant $C_2$ such that for any $m\geq C_2\underline{h}^{-d}$ and any quasi-uniform set $\{\eta_i^*\}_{i=1}^m$, the order-zero approximant \eqref{eqn_dls_domain} satisfies
    \begin{equation}\label{eqn_main_ls_determ_Omega}
    \begin{split}
        %&\|f-f_{n,m}\|_{\ell^2(X)}\lesssim n^{-\frac{d+2k+1}{2d}}\|f\|_{\mathcal{W}^{\frac{d+2k+1}{2},\infty}(\SS^d)},\\
        &\|f-f_{n,m}\|_{\mathcal{H}^s(\Omega)}\lesssim \underline h^{-s}h^{r}\left\{\begin{array}{ll}
            \|f\|_{\mathcal{W}^{r,p}(\Omega)}, &\qquad \frac dp<r\leq\frac{d}{2},  \\
            \|f\|_{\mathcal{H}^r(\Omega)}, &\qquad \frac{d}{2}<r\leq\frac{d+2k+1}{2},
        \end{array}\right.
    \end{split}
    \end{equation}
    where the corresponding constant is independent of $n,m$, $\Theta_n$, and $\{\eta_i^*\}_{i=1}^m$. In particular, if $\Theta=\{\Theta_n\}_{n\ge1}$ is an antipodally quasi-uniform sequence and $h=h_n$ corresponds to its $n$-th set, then with $m=\lceil C_2\underline h^{-d}\rceil$
    \begin{equation}
    \|f-f_{n,m}\|_{\mathcal{H}^s(\Omega)}\lesssim n^{-\frac{r-s}{d}}\left\{\begin{array}{ll}
            \|f\|_{\mathcal{W}^{r,p}(\Omega)}, &\qquad \frac dp<r\leq\frac{d}{2},  \\
            \|f\|_{\mathcal{H}^r(\Omega)}, &\qquad \frac{d}{2}<r\leq\frac{d+2k+1}{2},
        \end{array}\right.
    \end{equation}
\end{theorem}

\begin{proof}
By the $\beta=0$ specialization of Theorem~\ref{thm:leastsquare} applied on $\mathbb S^d$,
\[
\left\|
\mathcal{E}(f)-\mathcal{E}(f)_{n,m}
\right\|_{\mathcal H^s(\mathbb S^d)}\lesssim \underline h^{-s}h^r
\begin{cases}
\|\mathcal{E}(f)\|_{\mathcal W^{r,p}(\mathbb S^d)},\quad \frac dp<r\leq\frac{d}{2}\\
\|\mathcal{E}(f)\|_{\mathcal H^r(\mathbb S^d)},\quad\frac{d}{2}<r\leq\frac{d+2k+1}{2}.
\end{cases}
\]
Using \eqref{eq:extension_regular} gives the desired bound on the
spherical cap. Finally, applying $S_k$ and using Lemma~\ref{lem:SkTk},
\[
S_k
\left(
\sigma_k(\theta_j^*\cdot\circ)
\right)
=
\sigma_k(w_j^*\cdot x+b_j^*),
\]
which gives the ReLU$^k$ network representation on $\Omega$. The boundedness of $S_k$ on the cap, followed by restriction from $\BB^d$ to $\Omega$, completes the proof.
\end{proof}

Theorem~\ref{thm:leastsquareOmega} concerns the specific order-zero approximant obtained from spherical samples of the chosen extension $\mathcal E(f)$. It is therefore an approximation consequence of the spherical theorem, not a direct general-domain collocation theorem. For $\beta>0$, a bounded-domain PDE theory would additionally require a prescribed physical-domain operator, boundary conditions and boundary residuals, and sampling stability for the resulting interior and boundary trial spaces. These questions remain open.

\section{Numerical examples}\label{sec:numerical_experiments}
The purpose of this section is not to establish new empirical performance guarantees, but to illustrate the approximation behavior predicted by the theoretical results.

\subsection{Continuous least square examples}
We examine the deterministic least-squares method without multipliers ($\mathfrak L_\beta=I$) on \(\mathbb S^d\), \(d=2,3,4,5\), using activation orders \(k=1,2,2,3\), respectively. These choices all satisfy the strict hypothesis \(k>\frac{d-1}{2}\) in Theorem~\ref{thm:leastsquare}. The smooth parity-compatible targets are
\[
\begin{aligned}
 f^{(2)}(x)&=x_1x_2+\tfrac12(x_3^2-\tfrac13),&
 f^{(3)}(x)&=x_1x_2x_3,\\
 f^{(4)}(x)&=x_1x_2x_3,&
 f^{(5)}(x)&=x_1x_2x_3x_4 .
\end{aligned}
\]
Thus \(f^{(d)}(-x)=(-1)^{k+1}f^{(d)}(x)\) in every case. At the largest regularity allowed by the theorem and with \(s=0\), the reference exponents are \(5/4\), \(4/3\), \(9/8\), and \(6/5\).

For each \(n=64,128,256,512,1024\), the parameter set is selected by a max--min construction with explicit antipodal exclusion. The deterministic collocation set contains \(m=8n\) points and is selected independently by max--min sampling from a uniform sphere-point pool. Based on the (antipodally) quasi-uniformity assumptions of the theorem, the empirical mesh ratios remain bounded over the tested range: the largest estimated parameter ratio \( h_n/\underline h_n\) is \(3.368\), and the largest estimated collocation points ratio is \(1.722\).

The feature matrix is evaluated as
\[
 A_{ij}=\big(\max\{\theta_j\cdot\eta_i,0\}\big)^k,
\]
and the unregularized least-squares problem is solved by an SVD-based method. Relative \(\mathcal L^2\) errors are evaluated on independent scrambled Sobol point sets. We separately verified the vectorized powers \(k=1,2,3,4\) against scalar evaluation and recovered exactly generated ReLU, ReLU$^2$, and ReLU$^3$ networks on independent test points. These checks ensure that the higher-order activations required in dimensions \(d\ge3\) are implemented as ReLU$^k$, rather than as repeated ReLU layers or ordinary ReLU features.

Table~\ref{tab:least_squares_large_scale} reports the last three refinement levels. The error decreases monotonically in all four dimensions through \(n=1024\).

\begin{table}[H]
\centering
\small
\begin{tabular}{@{}rcccc@{}}
\toprule
\(n\) & \(\mathbb S^2\), ReLU & \(\mathbb S^3\), ReLU$^2$ &
\(\mathbb S^4\), ReLU$^2$ & \(\mathbb S^5\), ReLU$^3$\\
\midrule
256  & \(2.445\times10^{-3}\) & \(5.678\times10^{-3}\) & \(2.240\times10^{-2}\) & \(1.208\times10^{-1}\)\\
512  & \(9.583\times10^{-4}\) & \(2.094\times10^{-3}\) & \(8.818\times10^{-3}\) & \(3.590\times10^{-2}\)\\
1024 & \(4.259\times10^{-4}\) & \(8.671\times10^{-4}\) & \(3.845\times10^{-3}\) & \(1.202\times10^{-2}\)\\
\bottomrule
\end{tabular}
\caption{Relative \(\mathcal L^2\) errors for deterministic least squares with \(m=8n\).}
\label{tab:least_squares_large_scale}
\end{table}

Least-squares fits over \(n=256,512,1024\) give slopes \(-1.261\), \(-1.356\), \(-1.271\), and \(-1.664\) on \(\mathbb S^2\) through \(\mathbb S^5\), respectively. The first three agree closely with or exceed the reference rates over the tested range. The steeper ReLU$^3$ slope follows a visible pre-asymptotic transition and should not be interpreted as a sharper asymptotic result. All 20 design matrices have full column rank, the largest normalized normal-equation residual is \(6.39\times10^{-15}\), and independent \(2^{17}\)- and \(2^{18}\)-point error estimates differ by at most \(2.72\times10^{-3}\) relatively. These diagnostics verify the numerical solves and show that integration error is smaller than the changes across refinement levels.

\begin{figure}[H]
    \centering
    \includegraphics[width=0.82\textwidth]{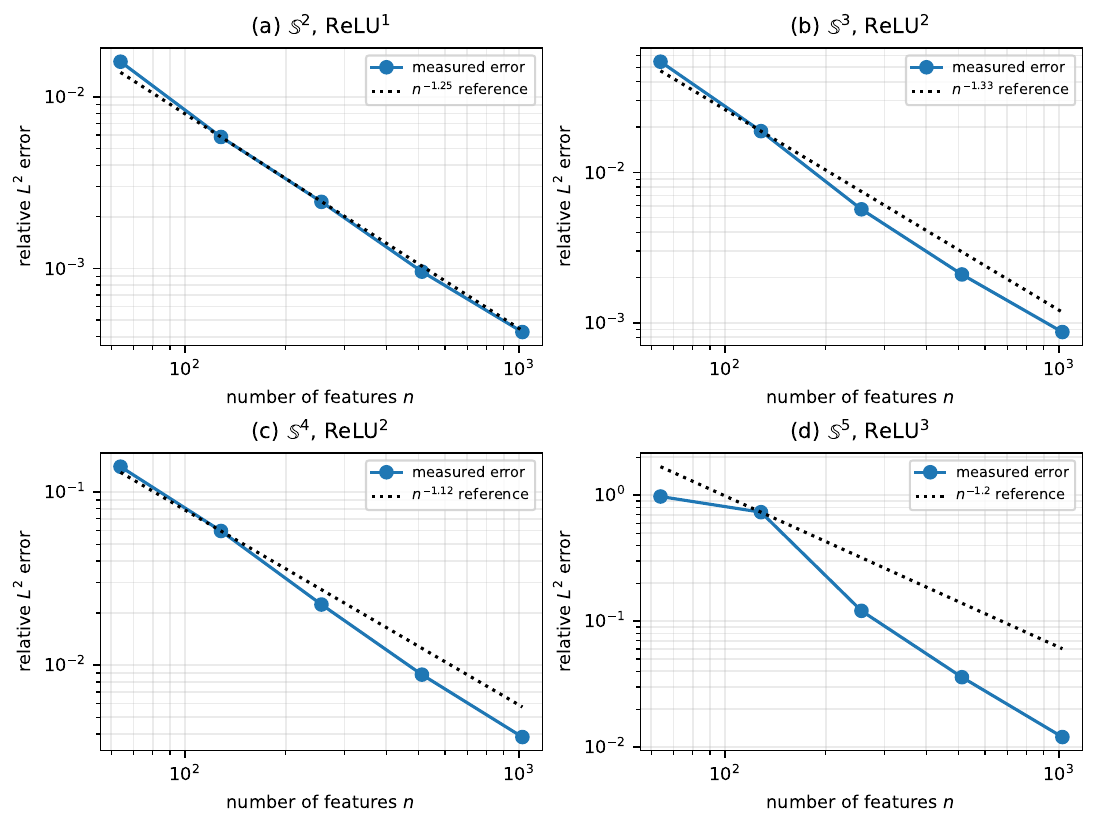}
    \caption{Deterministic least-squares errors on \(\mathbb S^2\)--\(\mathbb S^5\) with \(m=8n\). Each activation order satisfies \(k>\frac{d-1}{2}\); dotted lines show the reference exponents from Theorem~\ref{thm:leastsquare}.}
    \label{fig:least_squares_large_scale}
\end{figure}

The results are finite-range numerical evidence for the spherical theory, not an empirical proof of its asymptotic assumptions. Additional sampling parameters, random-collocation and oversampling studies, implementation checks, and solver diagnostics are recorded separately as reproducibility material.

\subsection{Positive-order residual examples}

We test the positive-order equation
\[
    (I-\Delta)u=f
\]
on $\mathbb S^2$ and $\mathbb S^3$. Thus $\beta=2$. For
$t=\theta\cdot x$ and $k\ge3$, the zonal Laplace--Beltrami identity gives
\begin{equation}\label{eq:pde_residual_feature}
\begin{split}
(I-\Delta)\sigma_k(\theta\cdot x)
={}&\bigl[1+k(k+d-1)\bigr](t_+)^k
-k(k-1)(t_+)^{k-2}.
\end{split}
\end{equation}
Consequently, the residual design matrix can be evaluated analytically,
without numerical differentiation. We verified
\eqref{eq:pde_residual_feature} independently both from the zonal derivative
formula
\[
\Delta\phi(t)=(1-t^2)\phi''(t)-dt\phi'(t)
\]
and by centered second differences along an orthonormal basis of geodesics
through randomly selected sphere points.

The two cases are
\begin{center}
\begin{tabular}{@{}cccc@{}}
\toprule
domain & \(k\) & \(u\) & \(f\)\\
\midrule
\(\mathbb S^2\) & \(3\) &
\(x_1x_2+\frac12(x_3^2-\frac13)\) & \(7u\)\\[0.2em]
\(\mathbb S^3\) & \(4\) & \(x_1x_2x_3\) & \(16u\)\\
\bottomrule
\end{tabular}
\end{center}
The solutions are spherical harmonics of degrees $2$ and $3$, respectively,
so that the forcing factors follow from
$-\Delta Y_{m,\ell}=m(m+d-1)Y_{m,\ell}$. Their parity agrees with
$(-1)^{k+1}$, and the strict pointwise-residual condition
$k>\beta+(d-1)/2$ holds in both cases. At $s=0$ and the largest forcing
regularity admitted by Theorem~\ref{thm:leastsquare}, the reference
exponents are $5/4$ on $\mathbb S^2$ and $4/3$ on $\mathbb S^3$.

For $n=64,128,256,512,1024$, we use independently selected max--min
parameter and collocation sets with explicit antipodal exclusion for the
parameters and $m=8n$. The residual least-squares problems use
\eqref{eq:pde_residual_feature}. Errors are evaluated on independent
scrambled Sobol sets with $2^{17}$ and $2^{18}$ points. Define the
operator-induced energy norm
\[
    \|v\|_{I-\Delta}:=\|(I-\Delta)v\|_{\mathcal L^2(\mathbb S^d)}.
\]
It is equivalent to the $\mathcal H^2$ norm by
\eqref{eqn:operator_equivalence}, but it is not identified with the
particular equivalent norm chosen in
\eqref{eqn:Sob_norm_Parseval}. The relative residual error in
Table~\ref{tab:least_squares_pde} is exactly the corresponding relative
$\|\cdot\|_{I-\Delta}$ solution error.

\begin{table}[H]
\centering
\small
\begin{tabular}{@{}rcccc@{}}
\toprule
&\multicolumn{2}{c}{$\mathbb S^2$, ReLU$^3$}
&\multicolumn{2}{c}{$\mathbb S^3$, ReLU$^4$}\\
\cmidrule(lr){2-3}\cmidrule(l){4-5}
$n$ & $E_{\rm res}$ & $E_{L^2}$
& $E_{\rm res}$ & $E_{L^2}$\\
\midrule
256  & $2.152\times10^{-3}$ & $1.306\times10^{-4}$
     & $4.014\times10^{-3}$ & $5.844\times10^{-4}$\\
512  & $8.324\times10^{-4}$ & $4.570\times10^{-5}$
     & $1.336\times10^{-3}$ & $1.304\times10^{-4}$\\
1024 & $3.688\times10^{-4}$ & $9.554\times10^{-6}$
     & $5.237\times10^{-4}$ & $6.944\times10^{-5}$\\
\midrule
fitted rate
     & \multicolumn{2}{c}{$-1.272$}
     & \multicolumn{2}{c}{$-1.469$}\\
reference rate
     & \multicolumn{2}{c}{$-1.250$}
     & \multicolumn{2}{c}{$-1.333$}\\
\bottomrule
\end{tabular}
\caption{Residual least-squares errors for $(I-\Delta)u=f$ with $m=8n$.
Here $E_{\rm res}$ and $E_{L^2}$ denote the relative residual and relative
solution $\mathcal L^2$ errors. The slopes are fitted over
$n=256,512,1024$.}
\label{tab:least_squares_pde}
\end{table}

All ten residual design matrices have full column rank. The largest
normalized normal-equation residual is $2.17\times10^{-15}$, the two
independent residual estimates differ by at most $2.44\times10^{-3}$
relatively, the largest estimated parameter ratio
$\widehat h/\underline h$ is $3.423$, and the largest estimated collocation
mesh ratio is $1.675$. These diagnostics verify the linear solves and show
that the observed changes across refinement levels are larger than the
independent integration discrepancy. The slopes are finite-range
observations, not a proof of asymptotic quasi-uniformity or of the theorem.

\begin{figure}[H]
    \centering
    \includegraphics[width=0.92\textwidth]{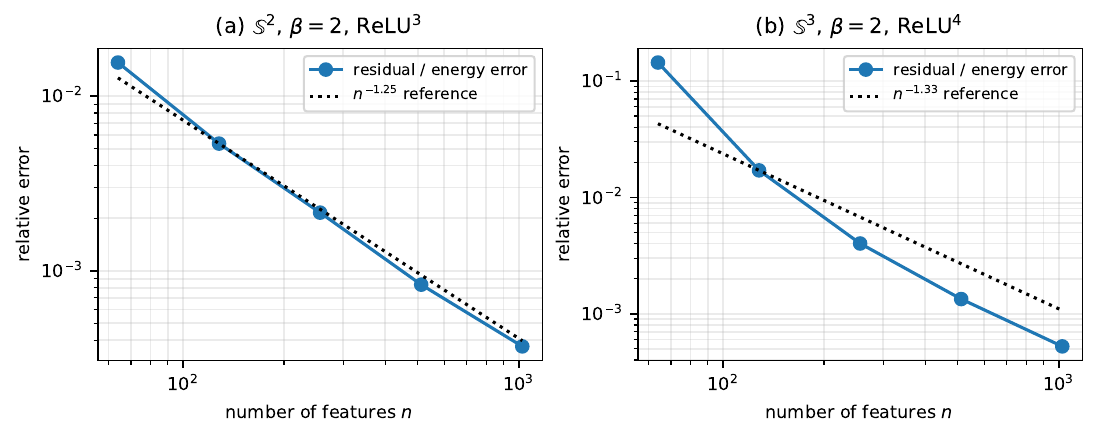}
    \caption{Relative residual, equivalently operator-induced energy, errors
    for $(I-\Delta)u=f$. Dotted lines show the deterministic reference
    exponents from Theorem~\ref{thm:leastsquare}.}
    \label{fig:least_squares_pde}
\end{figure}

\subsection{Examples on cubes}

We finally apply the same linear least-squares algorithm to ordinary affine networks on the hypercube $Q_d=[-1,1]^d$. This experiment is deliberately separated from the preceding validation: Theorem~\ref{thm:leastsquare} is a result on the sphere and provides no stability estimate or convergence rate for uniform samples in $Q_d$.

We consider $(d,k)=(2,1),(4,2),(8,4)$ and the smooth target
\[
    u_d(x)=d^{-1/2}\sum_{\ell=1}^d\sin(\pi x_\ell)
    +\frac14\prod_{\ell=1}^{\min\{4,d\}}x_\ell.
\]
The features are $\sigma_k(w_j\cdot x+b_j)$ with $(w_j,b_j)\in\mathbb S^d$. Parameters whose zero hyperplane does not intersect the central cube $[-1/2,1/2]^d$ are discarded, after which a max--min rule selects the feature directions. Collocation uses $m=8n$ scrambled Sobol points uniformly distributed in $Q_d$, and errors are evaluated on independent Sobol sets.

\begin{table}[H]
\centering
\small
\begin{tabular}{@{}rccc@{}}
\toprule
$n$ & $Q_2$, ReLU & $Q_4$, ReLU$^2$ & $Q_8$, ReLU$^4$\\
\midrule
128  & $3.664\times10^{-2}$ & $8.624\times10^{-2}$ & $7.251\times10^{-1}$\\
256  & $2.525\times10^{-2}$ & $5.008\times10^{-2}$ & $5.511\times10^{-1}$\\
512  & $1.967\times10^{-2}$ & $2.759\times10^{-2}$ & $2.952\times10^{-1}$\\
1024 & $1.572\times10^{-2}$ & $1.404\times10^{-2}$ & $1.426\times10^{-1}$\\
\midrule
fitted rate & $-0.402$ & $-0.872$ & $-0.794$\\
reference rate & $-1.250$ & $-1.125$ & $-1.0625$\\
\bottomrule
\end{tabular}
\caption{Relative $\mathcal L^2(Q_d)$ errors for affine ReLU$^k$ least squares with $m=8n$. The fitted rates use $n=128,256,512,1024$. The spherical reference rates are those supplied by Theorem~\ref{thm:leastsquare} for the same $(d,k)$ at $s=0$ and maximal admissible regularity; they are not theoretical rates for $Q_d$.}
\label{tab:least_squares_hypercube}
\end{table}

% The error decreases monotonically in all three cases. Least-squares fits of $\log\|u_d-u_{d,n,m}\|_{\mathcal L^2(Q_d)}$ against $\log n$ give the empirical rates $n^{-0.402}$, $n^{-0.872}$, and $n^{-0.794}$ for $d=2,4,8$, respectively. For comparison, on $\mathbb S^d$ Theorem~\ref{thm:leastsquare}, with $s=0$ and $r=\frac{d+2k+1}{2}$, gives the reference rate
% \[
%     n^{-\frac{d+2k+1}{2d}},
% \]
% namely $n^{-5/4}$, $n^{-9/8}$, and $n^{-17/16}$ for the same three pairs $(d,k)$. The cube slopes are therefore slower over the tested range. This comparison is only descriptive: the domains, sampling measures, and target classes differ, and the spherical proof does not apply to $Q_d$. All 12 design matrices have full column rank, the largest normalized normal-equation residual is $1.26\times10^{-14}$, and the two independent error estimates differ by at most $2.56\times10^{-3}$ relatively. Thus the computation is numerically verified, but the fitted exponents neither supply the missing general-domain sampling inequality nor constitute universal cube rates.

\section{Conclusion}

We established a discrete residual least-squares theory for positive elliptic spectral equations on the sphere using linearized ReLU$^k$ network trial spaces. The deterministic result gives the geometric error bound $\underline h^{-s}h^r$ for quasi-uniform collocation with $m\gtrsim\underline h^{-d}$. For antipodally quasi-uniform network parameter sets, this becomes the optimal rate $n^{-\frac{r-s}{d}}$ and agrees with the continuous approximation order. Through the elliptic norm equivalence, the residual estimate directly yields an error estimate for the approximate PDE solution. The order-zero choice $\mathfrak L_0=I$ recovers the ordinary discrete least-squares approximation problem.

The principal stability mechanism is the Bernstein inequality for linearized ReLU$^k$ spaces and its residual-space consequence for $\mathfrak L_\beta L_n^k$. The same framework gives a near-optimal high-probability estimate for i.i.d.\ uniform collocation points, while the original Bernstein inequality yields an inverse approximation theorem for the underlying network spaces.

The bounded-domain construction and the numerical experiments in the present paper remain restricted to the order-zero case. In particular, the spherical lifting does not identify $\mathfrak L_\beta$ with a differential operator on a bounded physical domain and does not account for boundary conditions. Developing a direct residual least-squares theory for bounded-domain PDEs, including boundary residuals, variable-coefficient operators, and sampling stability on the physical domain, remains an important direction for future work.

\bibliographystyle{abbrv}
	\bibliography{ref}
\end{document}